\documentclass{amsart}

\usepackage[normalem]{ulem}

\usepackage{color}
\usepackage{amssymb}

\allowdisplaybreaks

\numberwithin{equation}{section}

\newtheorem{theorem}{Theorem}[section]
\newtheorem{proposition}{Proposition}[section]
\newtheorem{lemma}[theorem]{Lemma}
\newtheorem{corollary}[theorem]{Corollary}
\newtheorem{definition}[theorem]{Definition}

\def\C{{\mathbb C}}
\def\Z{{\mathbb Z}}

\def\N{{\mathbb N}}

\def\R{{\mathbb R}}
\def\B{{\mathcal B}}
\def\P{{\bf P}}
\def\E{{\bf E}}
\def\L{{\bf L}}
\def\Q{{\bf Q}}
\def\I{{\bf I}}

\def\I{{\bf I}}

\def\<{\left<}
\def\>{\right>}
\def\dist{{\rm dist}}

\def\essup{{\rm outsup}}

\title[$L^p$ theory for outer measures]{
$L^p$ theory for outer measures and two themes
of Lennart Carleson united
}

\dedicatory{Dedicated to Lennart Carleson}
\author{Yen Do}
\address{Yen Do, Department of Mathematics,
Yale University, New Haven, CT 06511, USA}
\email{yen.do@yale.edu}

\author{Christoph Thiele}
\address{Christoph Thiele, Mathematisches Institut,
Universit\"at Bonn, Endenicher Alle 60, D-53115 Bonn, and 
Department of Mathematics,
UCLA, Los Angeles, CA 90095, USA}
\email{thiele@math.uni-bonn.de}


\thanks{Y.D. partially supported by NSF grant DMS 1201456.}
\thanks{C.Th. partially supported by NSF grant DMS 1001535.}
\date{\today}

\begin{document}

\begin{abstract}
We develop a theory of $L^p$ spaces based on outer measures
generated through coverings by distinguished sets. The 
theory includes as special case the classical $L^p$ theory on 
Euclidean spaces as well as some previously considered generalizations. 
The theory is a framework to describe aspects of singular integral theory 
such as Carleson embedding theorems, paraproduct estimates and $T(1)$ theorems.
It is particularly useful for generalizations of singular integral theory in 
time-frequency analysis, the latter originating in Carleson's investigation 
of convergence of Fourier series.  We formulate and prove a generalized Carleson 
embedding theorem and give a relatively short reduction of the most basic $L^p$
estimates for the bilinear Hilbert transform to this new Carleson embedding theorem.
\end{abstract}

\maketitle

\section{Introduction}\label{s.intro}

Two seminal papers of Lennart Carleson of the 1960's each introduced
a new tool into analysis that had profound influence. In his
paper \cite{carleson1}, {\it Interpolation by bounded analytic functions 
and the corona problem}, he introduced  what later became known 
as Carleson measures. Carleson measures revolutionized singular integral 
theory, where they are for example related to the space $BMO$, 
and related areas in real and complex analysis. In 
his celebrated paper \cite{carleson}, {\it On convergence and growth of 
partial sums of Fourier series}, Carleson introduced what we now call time-frequency analysis. 
Time-frequency analysis has remained until now an indispensable tool for 
its original application of
controlling Fourier series pointwise as 
well as a number of other applications including $L^p$ estimates for 
the bilinear Hilbert transform. Our present paper shows that a natural
$L^p$ theory for outer measures offers a unifying language for both 
Carleson measures and time-frequency analysis. The fundamental nature
of our $L^p$ theory for outer measures might in hindsight be an 
explanation for the important role of Carleson measures.

This paper is divided  into three parts. In the first part, Sections
\ref{outermeasure} and \ref{lptheory}, we carefully develop the basic 
$L^p$ theory for outer measure spaces. This part is in nature open ended 
and will hopefully lead to further investigations of outer measure spaces. 
We have focused only on those aspects of the theory that are directly relevant 
for the applications that we have in mind in the other parts of this paper.

Outer measures are subadditive set functions. In contrast to measures, 
outer measures do not necessarily satisfy additivity for disjoint finite 
or countable collections of sets. 
Some outer measures give rise to interesting measures by restriction to 
Caratheodory measurable sets, the most prominent example is classical 
Lebesgue theory. However, general outer measures need not give rise 
to interesting measures and one is led to studying outer measure spaces
for their own sake. Lacking additivity for disjoint sets one can not expect a 
useful linear theory of integrals with respect to outer measure. A good
replacement is a sub-linear or quasi sub-linear theory, which leads 
directly to norms or quasi norms rather than integrals. 
Naturally, $L^p$ norms are among the most basic norms to consider
in the context of outer measures.

There is a rich literature on outer measures, for example on capacity theory.
In contrast to previously developed theories based on the Choquet integral, 
we do not in general base our $L^p$ theory on the outer measure of super level 
sets $\{x:f(x)>\lambda\}$ for a function $f$. Instead, we use a more subtly 
defined quantity (Definition \ref{superleveldefinition}) to replace the outer 
measure of a super level set. This new quantity, which we call 
{\it super level measure}, involves pre-defined averages 
over the generating sets of the outer measure. If the pre-defined averages 
are of $L^\infty$ type, the super level measure specializes to the outer 
measure of the super level set, but in general the two quantities are 
quite different. Once we have introduced the super level measure, 
the $L^p$ theory develops in standard fashion, and we 
develop it to the extend that we need for subsequent parts of the paper.

In the second part of this paper, Section \ref{t1section}, we describe
how outer measures can be used in the context of
Carleson measures. It is our first example of an outer measure space 
in which our refined definition of super level measure does not coincide
with the classical case of the outer measure of super level set. 
The outer measure space in question
is the upper half plane and the outer measure is generated by tents. 
The essentially bounded functions with respect
to the outer measure in this upper half plane are Carleson measures.
Moreover, the identification of a function on the boundary with the harmonic 
extension in the enterior of the upper half plane, that is the Carleson 
embedding map, turns out a basic example of a bounded map from a classical 
$L^p$ space to an outer $L^p$ space. 
We describe in Section \ref{t1section} how classical estimates for
paraproducts and $T(1)$ theorems can be proved by an outer H\"older
inequality together with such embedding theorems. In this setting, the use
of outer $L^p$ spaces is very much in the spirit of the use of tent spaces 
introduced in \cite{coifmanms}. It is an artifact in this particular
situation that our notion of outer measure may be replaced 
with more classical concepts.

The full power of the new outer $L^p$ spaces becomes evident in 
its applications in time-frequency analysis, that we discuss in the 
third part of this paper.  The underlying space for the outer measure
becomes the Cartesian product of the upper half plane with a real
line. In this setting there are no evident analogues of the tent spaces of
\cite{coifmanms}
that one could use in place of outer $L^p$ spaces.  We formulate and prove 
a novel generalized Carleson embedding theorem, Theorem \ref{gen.carl.emb}, in Section 
\ref{gentents}. It is a compressed and elegant way to state an essential 
part of time-frequency analysis. In Section \ref{bhtsection} we then 
use the generalized Carleson embedding theorem to reprove bounds for 
the bilinear Hilbert transform.

The generalized Carleson embedding theorem can also be used
as an ingredient to prove almost everywhere convergence of partial
Fourier integrals of $L^p$ functions with $2<p<\infty$. One would need
an additional Carleson embedding theorem, either analoguos to the 
interplay between energy and mass in \cite{lacey-thiele}, or analoguos 
to some vector valued version of the Carleson embedding theorem as in 
\cite{demeter-thiele}. We also envision the generalized Carleson embedding 
theorem and variants thereof to be useful in further advances in time-frequency
analysis. We were led to the theory of outer $L^p$ spaces while working
on variation norm estimates as in \cite{oberlin-et-al} in the setting of 
biest type operators as in \cite{MTT-BiestFourier}. For brevity of the 
present paper, and because of the various possible routes towards Carleson's 
theorem, we decided to restrict this exposition to a discussion of the bilinear Hilbert transform. This already captures many essential parts of Carleson's
time-frequency analysis.

Gaining a streamlined view on time-frequency analysis was the original 
motivation for the present paper, which is the outcome of a long evolution 
process.
In traditional time-frequency analysis, one proves bounds of multilinear forms passing 
through model sums
$$\Lambda=\sum_{P\in \P} c_P \prod_{j=1}^n a_j(P)\ \ ,$$
where the summation index runs through a discrete set, typically
a collection of rectangles (tiles) in the phase plane. 
The coefficients $c_P$ are inherent to the multilinear form, while the 
sequences $a_j$ each depend on one of the input functions for the multilinear
form in question. There is a multitude of examples in the literature for the
tile sequences $a_j$, the most basic example being normalized wave packet coefficients
\begin{equation}\label{mostbasic}a_j(P)=\<f,\phi_P\>
\end{equation}
for the $L^1$ normalized wave packets
$$\phi_P(x) = 2^{-k}\phi(2^{-k}x-n)e^{2\pi i 2^{-k} x l}\ \ ,$$
where $k,n,l$ are integers and parameterize the space $\P$,
and $\phi$ is a suitably chosen Schwartz class function.
These coefficients are much in the spirit of the embedding maps 
considered in Sections \ref{gentents} and \ref{bhtsection} of the
present paper. Use of such wavepackets in the study of the bilinear 
Hilbert transform appears in \cite{lacey-thiele1}.
In the dyadic model as in \cite{thiele}, one defines wave packets with 
respect to abstract Fourier analysis on the group  $\Z/2\Z$. More 
generally one can have tile semi-norms 
\begin{equation}\label{e.standard-sequence}
a_j(P)=\sup_{\phi\in \Phi}|\<f,\phi_P\>|
\end{equation}
where one maximizes and possibly also averages over a suitably chosen 
set $\Phi$ of generating functions. This approach has
been useful in \cite{thiele-uniform} and more explicitly
in \cite{MTT-uniform}. To prove bounds on Carleson's operator,
\cite{lacey-thiele} uses modified wave packets 
\begin{equation}\label{e.carleson-sequence}
a_j(P)=\<f,\phi_P1_{\{(x,N(x))\in P\}}\>
\end{equation}
for the linearizing function $N$ of the linearized Carleson operator.
In some instances such as in \cite{MTT-BiestFourier}, the definition of $a_j(P)$
may involve itself a multi-linear operator whose analysis requires another 
level of time-frequency analysis. For variational estimates of the Carleson 
operator as in \cite{oberlin-et-al}, one has variational wave packets
\begin{equation}\label{e.varcarleson-sequence}
a_j(P) = \<f,\phi_P \sum_{k} v_k1_{\{(x,N_k)\in P_2, (x,N_{k-1})\not\in P\}}\>
\end{equation}
for a sequence of linearizing functions  $N_0(x)<N_1(x)<\dots$ and a sequence of dualizing functions $v_1(x),v_2(x),\dots \ \ ,$ such that for some $r>2$ we have the uniform bound
$$\sum_{k} |v_k(x)|^{r'} = O(1) \ \ .$$
A point of the present paper is that in many of these examples the bound 
on $\Lambda$ is a H\"older inequality with respect to
an outer measure on the space $P$ :
$$|\Lambda|\le C \sup_{P\in P}|c_P| \prod_{j=1}^n \|a_j\|_{L^p_j(\P,\dots)}$$
where the dots stand for specifications of the outer measure structures 
in each example. The rest of the proof of boundedness of $\Lambda$ then 
becomes modular in that one has to prove bounds for each $j$ separately 
on the  outer $L^p$ norms of the sequences $a_j$, estimates which take for 
example the form $$\|a_j\|_{L^p_j(\P,\dots)}\le \|f_j\|_p\ ,$$
where $f_j$ may be the corresponding input function to the original
multilinear form as for example in \eqref{mostbasic}, and the 
$L^p$ norm is in the classical sense.

A novelty in the present paper is that we do not have to pass through
a discrete model form, but rather work with an outer measure space
on a continuum. This avoids both the cumbersome introduction of the discrete
spaces as well as the usual technicalities in the discretization process.

The factorization of the multilinear form in time-frequency analysis 
into embedding theorems on the one hand and an outer H\"older's 
inequality on the other hand 
is a clear modularization of the matter and promises to be useful in 
other applications of time-frequency
analysis. Indeed, we were explicitly studying the modularization
process because with Camil Muscalu we were considering a program
outlined in \cite{do-muscalu-thiele} of estimating multilinear 
forms with nested  levels of time-frequency analysis. 

We are grateful to Mariusz Mirek for carefully reading
an early version of this manuscript and pointing out many corrections.
We are grateful to Pavel Zorin-Kranich for pointing out many corrections
and an error on the last pages of a previous version posted on arxiv, 
that was overcome by proving a generalized Carleson embedding theorem 
with parameters $\alpha$ and $\beta$ and by removing any claim about explicit 
dependence on such parameters in the theorem on the bilinear Hilbert transform.
Thanks to much stronger known uniform estimates for the bilinear Hilbert transform as in
\cite{grafakos-li}, tracking of the dependence on these parameters in our
proof was not our key point. 
We also thank Yumeng Ou for pointing out an error in a previously  
posted version that was overcome by changing the exponent of $\beta$
in Corollary \ref{tentcorollary}, with some minor impact on the rest of
the section. We thank an anonymous referee for a valuable list of
suggestions to improve this exposition.
We are grateful to Stefan M\"uller, Alexander Volberg, Igor Verbitzky 
and Nguyen Cong Phuc  for discussions on capacity theory, which is a much studied 
example for outer measures.
We finally are grateful for much feedback on outer 
measures during a season of conferences in which the ideas of this present 
paper were announced, and for the many suggestions on this exposition that
we have received.

\section{Outer measure spaces}\label{outermeasure}

\subsection{Outer measures}

An outer measure or exterior measure on a set $X$ is a monotone and 
subadditive function on the collection of subsets of $X$ with values in the extended nonnegative
real numbers, and with the value $0$ attained by the empty set.  

\begin{definition}[Outer measure]
Let $X$ be a set. An outer measure on $X$ is a function
$\mu$ from the collection of all subsets of $X$ to $[0,\infty]$
that satisfies the following properties:
\begin{enumerate}
\item 
If $E\subset E'$ for two subsets of $X$, then
$\mu(E)\le \mu(E')$.
\item $\mu(\emptyset)=0$.
\item
If $E_1,E_2,\dots$ is a countable collection of sets in $X$, then
\begin{equation}
\label{countsubadd}
\mu(\bigcup_{j=1}^\infty E_j)\le \sum_{j=1}^\infty \mu(E_j)\ .
\end{equation}
\end{enumerate}
\end{definition}

In the examples we have in mind, the space $X$ is an infinite complete metric space
and thus uncountable. The set of all subsets of $X$ has then even larger cardinality than the continuum, 
and can only be organized in abstract ways. The description of an outer measure then typically
comes in two steps: First one specifies concretely a quantity that we may call pre-measure 
on a small collection of subsets, and then one passes abstractly from the pre-measure to 
the outer measure by means of covering an arbitrary subset by sets in the small collection.
This covering process is the intuition behind the adjective {\it outer} in the term {\it outer measure}.

\begin{proposition}[Abstract generation of outer measure by a concrete pre-measure]\label{generate.outer}
Let $X$ be a set and $\E$ a collection of subsets of $X$. 
Let $\sigma$ be a function from $\E$ to $[0,\infty)$.
Define for an arbitrary subset $E$ of $X$ 
$$\mu(E):=\inf_{\E' }\sum_{E'\in \E'} \sigma(E')\ ,$$
where the infimum is taken over all countable subcollections $\E'$ of $\E$
which cover the set $E$, that is whose union contains $E$. Here we understand that an empty sum is $0$.
Then $\mu$ is an outer measure. 
\end{proposition}

The concrete pre-measure requires the data $\E$, and $\sigma$. 
For simplicity we will often omit explicit mention of $\E$, since 
$\E$ is implicitly determined as the domain of $\sigma$.
The proof of the proposition is basic and standard, we 
reproduce it here for emphasis.
\begin{proof}
We need to prove the three defining properties of outer measures.

The empty collection of subsets covers the empty set, which shows 
$\mu(\emptyset)= 0$ since the empty sum of nonnegative numbers is $0$. 

If $F\subset F'$ for two subsets of $X$, then every cover of
$F'$ is a cover of $F$ and hence $\mu(F)\le \mu(F')$.
Let $F_1,F_2,\dots $ be a countable collection of subsets of $X$
and pick $\epsilon>0$. Find for each $i$ a countable subcollection $\E_i$ of $\E$ 
which covers $F_i$ and satisfies $$\sum_{E\in \E_i} \sigma(E)\le \mu(F_i)+ \epsilon 2^{-i}\ .$$ 
Then the union $\E'$ of the collections $\E_i$ covers the union of the sets $F_i$ and satisfies
$$\sum_{E\in \E'} \sigma(E) \le (\sum_i \mu(F_i))+ \epsilon\ . $$

Since $\epsilon$ was arbitrary, we conclude that $\mu(\bigcup F_i)\le \sum_i \mu(F_i)$. 
\end{proof}

It is in general not true that for $E\in \E$ we have $\sigma(E)=\mu(E)$, however
this identity can be established in many examples in practice. Clearly this identity 
holds precisely if for every set $E\in \E$ and every cover of $E$ by a countable
subcollection $\E'$ of $\E$, we have
\begin{equation}\label{ecoverc}
\sigma(E)\le \sum_{E'\in \E'}\sigma(E')\ .
\end{equation}
Then the most efficient cover of $E$ is by the trivial collection$\{E\}$,
which establishes $\sigma=\left. \mu\right|_\E$.

We did not allow $\sigma$ to take value $\infty$. This is no restriction,
since if we had $\sigma(E')=\infty$ for some $E'\in \E$, then using 
the set $E'$ in any cover of $E$ will make the sum 
$\sum_{E'\in \E'} \sigma(E')$ equal to $\infty$, a value that is as already 
the default even if no cover of $E$ exists at all. 

If the collection $\E$ is countable, the contribution of sets $E\in \E$ with $\sigma(E)=0$
trivializes. Namely, we may consider the union $E_0$ of the countably many generating
sets with pre-measure $0$. Then $E_0$ has outer measure zero, and we can construct
an outer measure on $X\setminus E_0$ which reflects the structure of the outer measure on $X$
but does not contain any generating set with pre-measure $0$.

\subsection{Examples for outer measures}
\label{examplesection}

\subsubsection*{Example 1: Lebesgue measure via dyadic cubes}
Let $X$ be the Euclidean space $\R^m$ for some $m\ge 1$ and let 
$\E$ be the set of all dyadic cubes, that is all cubes of the form
$$Q=[2^{k}n_1, 2^{k}(n_1+1))\times \dots \times [2^{k}n_m, 2^{k}(n_m+1))$$
with integers $k,n_1,\dots ,n_m$.
For each dyadic cube $Q$ we set
$$\sigma(Q)=2^{mk}\ .$$
Then $\sigma$ generates an outer measure which is the classical Lebesgue 
outer measure on $\R^m$. We have $\sigma(Q)=\mu(Q)$ for every dyadic cube.
This latter fact requires a bit of work, in fact it is one of the 
more laborious items in the standard introduction of Lebesgue measure.

\subsubsection*{Example 2: Lebesgue measure via balls}
Let $X=\R^m$ as above and let $\E$ be the set of all open balls
$B_r(x)$ 
with 
radius
 $r$ and 
center 
$x\in \R^m$.
Let $\sigma(B_r(x))=r^m$ for each such ball.
Then $\sigma$ generates a multiple of Lebesgue outer measure,
and again we have $\sigma=\left. \mu\right|_\E$.

If one desires a countable generating set, one may restrict
the collection of generating sets to the collection of balls
which have rational radius and rational center. This choice 
will result in the same outer measure.

\subsubsection*{Example 3: Outer measure generated by tents} 
Let $X=\R\times (0,\infty)$ be the open upper half plane 
and let $\E$ be the set of tents, that is open isosceles triangles 
of the form (see Figure 1 in Section \ref{t1section})
$$T(x,s)=\{(y,t)\in \R\times (0,\infty): t<s, |x-y|<s-t\}$$
for some pair $(x,s)\in \R\times (0,\infty)$ which describes 
the tip of the tent. Note that the constraint $t<s$ is implied by the
constraint $|x-y|<s-t$, but it is kept for emphasis. 
Define $\sigma(E)=s$ for any such tent, and note
that $\sigma(E)$ is equal to $\frac 12\sigma_L(\pi(E))$ where $\pi(E)$ is 
the projection of $E$ onto the first coordinate and 
thus an open ball in $\R$, and $\sigma_L$ is the generator of
Lebesgue outer measure on $\R$
described in Example $2$.

By projection onto the first coordinate
it easily follows from Example $2$ that $\mu$ satisfies (\ref{ecoverc}).
Again one obtains the same outer measure restricting the collection
of generating sets to the tents with rational tip.

\subsubsection*{Example 4: Capacity}
We restrict attention to a particular example of capacity, more examples
can be found in the survey \cite{adams}.
Let $X=\R^n$ with $n\ge 3$ and let $\E$ be the collection of open sets in $X$.
Define the kernel  $K(x):=|x|^{2-n}\ ,$ which is a multiple of the
classical Newtonian kernel.
Let $\sigma$ assign to each open set its capacity with respect to $K$,
that is the least upper bound for the total mass $\|\nu\|$ of a
positive Borel measure $\nu$ which has compact support in $E$ and 
satisfies $\|\nu*K\|_\infty\le 1$. Note that $\sigma(E)>0$ for every nonempty 
open set $E$, this can be seen  
by testing with
a measure $\nu$ associated with a smooth nonnegative density supported in 
a small compact ball contained in $E$.  

To see Property (\ref{ecoverc}),
assume $E$ is some open set covered by a countable collection $\E'$
of open sets. 
Let $ \nu$ be a measure supported on
a compact set $F\subset E$ such that $\|\nu*K\|_\infty\le 1$. Then
$$\|\nu\|\le \sum_{E'\in \E'}\|\nu 1_{E'}\|$$
$$\le 
\sum_{E'\in \E'}\sup_{F\subset E'}
\|\nu 1_F\|$$
$$\le 
\sum_{E'\in \E'} \sigma(E')  \sup_{F\subset E'}
\|(\nu 1_{F})*K\|_{\infty}$$
$$\le \sum_{E'\in \E'} \sigma(E')
\|\nu*K\|_\infty\le  \sum_{E'\in \E'} \sigma(E')\ . $$
Since $\E'$ was arbitrary, this proves $(\ref{ecoverc})$.

\subsection{Remarks on measurable sets}

Outer measures are used in classical textbooks such as \cite{wz} as a 
stepping stone towards the introduction of measures.
In measure theory, one is interested in equality in \eqref{countsubadd}
under the additional assumption that the sets $E_i$ are pairwise disjoint. 
Such equality does not follow in general from the properties of outer measure.
A sufficient additional criterion is that each of the sets $E_i$ is 
measurable, as in the following definition.

\begin{definition}[Measurability]
Let $\mu$ be an outer measure on a set $X$ generated by a pre-measure
on a collection $\E$.
An arbitrary subset $F$ of $X$ is called measurable  if for every generating set $E\in \E$ we have
$$\mu (F\cap E ) + \mu(F^c\cap E) = \mu(E) \ .$$
\end{definition}

We note that if $F$ is measurable, then it also satisfies the Caratheodory
criterion that for arbitrary subset $G$ of $X$ we have
$$\mu (F\cap G ) + \mu(F^c\cap G) = \mu(G)\ .$$
We briefly sketch the argument. If $\mu(G)$ is infinite, then it is easy to see that one of the
outer measures on the left hand side has to be infinite as well.
If $\mu(G)$ is finite, pick $\epsilon>0$ and a cover $\E'$ of $G$
by generating sets such that
$$\sum_{E\in \E'} \sigma(E)\le \mu(G)+\epsilon\ .$$
Then we have
$$\mu(G)\le \mu(F\cap G)+\mu(F^c\cap G)\le \sum_{E\in \E'}\mu(F\cap E)+\sum_{E\in \E'}\mu(F^c\cap E)$$
$$\le \sum_{E\in \E'} \mu(F\cap E)+\mu(F^c\cap E)=\sum_{E\in \E'} \mu(E) \le \sum_{E\in \E'} \sigma(E)\le \mu(G)+\epsilon\ .$$
Since $\epsilon$ was arbitrary, it follows that the first inequality in this line of reasoning
is indeed an equality.

In Example 1 above the measurable sets are called
Lebesgue measurable. To see existence of many Lebesgue measurable sets,
one observes that dyadic cubes are Lebesgue measurable. This follows from 
two observations: First one may estimate the outer measure of $F$ by coverings
with cubes of side length at most that of the given cube $E$. Second, each 
such small cube is either contained in $E$ or disjoint from $E$ allowing
to split the covering into two disjoint collections, of which one covers 
$F\cap E$ and the other covers $F\cap E^c$. 

One can show in general that the collection of measurable sets 
is closed under countable union and countable intersection, thus from
Lebesgue measurability of dyadic cubes one can conclude Lebesgue measurability
of all Borel sets in $\R^m$.  

In contrast, no set other than $\emptyset$ and $X$ is measurable in 
Example 3. For assume we are given a nontrivial subset $E$ of $X$, let
$(x_0,s_0)$ be a point in the boundary of $E$ and consider a tent
$T(x,s)$ which contains $(x_0,s_0)$ and satisfies $s<2s_0$.
Then we find points $(y,t)\in E\cap T(x,s)$ and $(y',t')\in E^c\cap T(x,s)$
in the vicinity of $(x_0,s_0)$ such that $s<t+t'$. Then we have
$$\mu(T(x,s))=\sigma(T(x,s))=s<t+t'\le \mu(T(x,s)\cap E)+\mu(T(x,s)\cap E^c)\ ,$$
where we used that if a set $F$ contains a point $(y,t)$, then $\mu(F)>t$ 
because any cover of $F$ needs to contain a tent with height at least $t$.
The last display shows that the set $E$ is not measurable.

In Example 4, it is well known that no bounded open set $E$ is measurable.

Namely, let $E_1$ and $E_2$ be disjoint bounded open sets  such that $\dist(E_1,E_2) > 0$ 
and set $E:=E_1\cup E_2$. Let $\nu$ be a  positive Borel  measure on a compact subset
of $E$ with $\|\nu*K\|_\infty \le 1$. Since $E$ is bounded,   for some finite constant $M$ that depends on the diameter of $E$ and $n$ it holds that
$$\|\nu\| \le M \inf_{x\in E}(\nu*K)(x) \le M \|\nu*K\|_{\infty} \le   M \ \ .$$

In particular, it follows that $\sigma(E),\sigma(E_1),\sigma(E_2) < \infty$ (they are positive from a previous discussion). Then, by inner regularity
of Borel measures, we obtain
$$\|\nu\|= \sum_{j=1}^2 \|\nu1_{E_j}\|\le \sum_{j=1}^2 \sigma(E_j)\|(\nu 1_{E_j})*K\|_{\infty}  \ \ .$$
Using the fact that $\sigma$ satisfies the countably subadditive property \eqref{ecoverc}, we obtain $\mu(E_j)=\sigma(E_j)$. Using harmonicity of $\nu1_{E_j}*K$ in the interior of $E_j^c$, it follows that
$$\|\nu\|\le \sum_{j=1}^2 \mu(E_j)\|(\nu 1_{E_j})*K\|_{\infty} $$
$$\le \sum_{j=1}^2 \mu(E_j)\|(\nu 1_{E_j})*K\|_{L^\infty(\overline{E_j})}$$
$$\le \sum_{j=1}^2 \mu(E_j)(\|\nu*K\|_{\infty} -\inf_{x\in \overline{E_j}} ((\nu 1_{E_{3-j}})*K)(x))$$
$$\le \sum_{j=1}^2 \mu(E_j) -
\sum_{j=1}^2 \mu(E_j)\inf_{x\in \overline{E_j}} ((\nu 1_{E_{3-j}})*K)(x)\ .
$$
Since $E$ is bounded, it follows that for some finite positive constant $M$ that depends on the diameter of $E$ and $n$ we have
$$\inf_{x\in \overline{E_j}}((\nu1_{E_{3-j}})*K)(x) \ge \frac 1 M \|\nu1_{E_{3-j}}\|$$
It follows that 
$$\|\nu\| \le  
\sum_{j=1}^2 \mu(E_j) - \sum_{j=1}^2 \mu(E_j)\frac 1{M} \|\nu1_{E_{3-j}}\| $$
$$\le \sum_{j=1}^2 \mu(E_j) - \frac 1 M \min(\mu(E_1),\mu(E_2)) \sum_{j=1}^2 \|\nu1_{E_{3-j}}\|$$
$$\le \sum_{j=1}^2 \mu(E_j) - \frac 1 M \min(\mu(E_1),\mu(E_2))  \|\nu\|$$
Since $\mu(E_1)>0$ and $\mu(E_2)>0$, it follows that for some constant $c>0$ that depends only on $E_1,E_2,n$ it holds that
$$\|\nu\| \le \frac 1{1+c} \sum_{j=1}^2 \mu(E_j)$$
Taking supremum over all such $\nu$ it follows that $\mu(E)=\sigma(E)<\mu(E_1)+\mu(E_2)$, thus
neither $E_1$ nor $E_2$ is measurable.
While in Example 3 the lack of measurable sets is intuitively caused by
the scarceness of the collection of generating sets, the collection $\E$ 
in this example is very rich and can hardly be blamed for the shortage of 
measurable sets.
 
\subsection{Functions and sizes}

We  propose an $L^p$ theory for functions on outer measure spaces.
One possible way of introducing an $L^p$ norm of a nonnegative function $f$
and $1\le p<\infty$ is via the following definition:
\begin{equation}\label{choquet}
(\int_0^{\infty} p\lambda^{p-1} \mu(\{x\in X: f(x)>\lambda\})\, d\lambda)^{1/p}\ .
\end{equation}
In many instances, this is the correct definition. However, we propose
a different formula, which in many examples such as Lebesgue
theory coincides with the above, but differs in full generality.
The motivation for our definition is that it appears more useful in 
the applications that we have in mind.

Our different approach already finds a motivation in the efficiency of encoding 
of functions in classical Lebesgue theory. Classical coding describes functions 
as assignment of a value to every point in the space $X$. For an $L^p$ function
this assignment has to be consistent with the measurability structure.  The set 
of such assignments has a very large cardinality, which is only reduced after 
consideration of equivalence classes of $L^p$ functions. This detour over sets of 
large cardinality can be avoided by coding functions via their averages over dyadic 
cubes. There are only countably many such averages, and by the Lebesgue Differentiation 
theorem these averages contain the complete information of the equivalence 
class of the $L^p$ function.

Unlike in the above definition of $L^p$ norm, which regards the function
$f$ as a pointwise assignment, we propose to build the $L^p$ theory on outer measure
spaces via averages over generating sets. The theory then splits again into a concrete
and abstract part, parallel to the construction of outer measures by 
generating sets. There will be a concrete procedure to assign to a function averages over 
generating sets, and further on there will be an abstract procedure to define the
$L^p$ norms of functions from such averages.
The concrete averaging procedure itself is based on some other measure theory
(which by itself might be an outer measure theory, but in the current paper 
we will not delve into such higher level iteration of the theory). We will
consider this other measure theory as concrete external input into the outer 
measure theory, while the genuine part of the outer measure theory is 
the  abstract passage from the concrete averages to outer $L^p$ norms.

The class of functions that we will be able to take $L^p$ norms of will
depend on the concrete averaging procedure we choose. To avoid too
abstract a setup we shall assume that $X$ is a metric space, and that 
every set of the collection $\E$ is Borel. We shall assume the concrete
averaging procedure will allow to average positive functions in the class 
$\B(X)$, 
the set of Borel measurable functions on $X$.
If the set $X$ is countable, a case that exhibits many of the 
essential ideas of the theory, the space $\B(X)$ is the space
of all functions on $X$.

As linearity is closely related with measurability, in the absence of
measurability we will not require averages to be linear but merely sub 
linear or even quasi sub linear. We will call these averages "sizes".

\begin{definition}[Size]\label{d.size}
Let $X$ be a metric space. Let $\sigma$ be a function on a collection
$\E$ of Borel subsets of $X$ and let $\mu$ be the outer measure generated by $\sigma$.
A size is a map
$$S: \B(X)\to [0,\infty]^\E$$
satisfying for every $f,g\in \B(X)$ and every $E\in  \E$
the following properties:
\begin{enumerate}
\item Monotonicity: if $|f|\le |g|$, then 
$S(f)(E)\le S(g)(E)$.
\item Scaling: $S(\lambda f)(E)= |\lambda| S(f)(E)$ for every $\lambda\in \C$.
\item Quasi-subadditivity: 
\begin{equation}\label{quasi.sub}
S(f+g)(E)\le C[S(f)(E)+S(g)(E)]
\end{equation}
for some constant $C$ depending only on $S$ but not on
$f,g,E$.
\end{enumerate}
\end{definition}

Note that (1) above implies $S(f)(E)=S(|f|)(E)$ for all $f$ and $E$.
Hence our theory is essentially one of  nonnegative  functions, and the
size needs initially be only defined for  nonnegative 
Borel functions and can then be extended via the above identity
to all functions.

We discuss sizes for Examples 1 through 4, and give a number of
forward looking remarks on particular aspects of the outer $L^p$
theory to be developed.

In Lebesgue theory in Example 1, we define for every Borel function $f\in \B(X)$
and every cube $Q$
$$S(f)(Q)=\mu(Q)^{-1}\int_Q |f(x)|\, dx\ .$$
The integral is in the Lebesgue sense. 
Note the coincidence that the 
measure theory used to define the size is the same as the measure 
theory associated with the outer measure $(X,\mu)$. This coincidence
is a particular feature of Example 1 (and 2 below).  The circularity of
this setup does not invalidate our theory, certainly Lebesgue measure
can be introduced without reference to the outer integration theory
that we develop in this paper.

Note that $S(f)(Q)$ is finite for every locally integrable function on 
$\R^m$. For such function we may define the ``martingale''
$$M(f)(Q):=\mu(Q)^{-1} \int_Q f(x)\, dx\ .$$
A consistency condition applies for $M(f)$, namely, the value of $M(f)$ 
on a dyadic cube is  equal to the average of the values 
on the dyadic subcubes of half the sidelength.  By the dyadic Lebesgue 
Differentiation theorem, the martingale uniquely determines 
the value of the function $f$ at every Lebesgue point, and this uniquely
determines the equivalence class of the measurable function $f$ in
Lebesgue sense. As noted before, the martingale is a very efficient way of 
encoding the function $f$. The space $L^\infty(\R^m)$ 
can be described as all bounded maps from $\E$ to $\C$ which satisfy the 
consistency condition. This example is a strong indication that a useful 
general theory of outer measure may be built out of assigning values
to elements $E\in \E$. Indeed, it would be possible in this example
to built the theory entirely out of maps $M:\E\to \C$
satisfying the consistency condition, without reference to any 
Borel function $f$.

Turning to Example 2, we may similarly define
$$S(f)(B)=\mu(B)^{-1}\int_B |f(x)|\, dx $$
for every ball $B$. Again, these averages determine
$f$ by the Lebesgue Differentiation theorem. In this case
there does not exist an easy algebraic consistency condition
that identifies maps from $\E$ to $\C$ that arise from
locally integrable functions $f$ as the average
$$M(f)(B):=\mu(B)^{-1} \int_B f(x)\, dx\ .$$
This provides the evidence that it is impracticable to 
build a theory of functions on outer measure space entirely
out of maps from $\E$ to $\C$ and without reference to a function $f$.

In Example 3 we make an assignment of a value to each tent by averaging
a Borel measurable function on the tent: 
\begin{equation}\label{primitivetentsize}
S(F)(T(x,s)):=s^{-1} \int_{T(x,s)} |F(y,t)|\, dy\, \frac{dt}{t}\ .
\end{equation}

This averaging is based on weighted Lebesgue measure on $X$, which however
is not the outer measure $(X,\mu)$ in this Example 3.
In the literature, one often works with the class of Borel measures
$\nu$ on $X$ rather than the class of Borel measurable functions, and 
defines
$$S(\nu)(T(x,s)):=|s|^{-1} |\nu| (T(x,s))\ \ .$$
If the function $S(\nu)$ is bounded, the 
measure $\nu$ is called a Carleson measure in the literature, the 
concept of which dates back to the seminal paper \cite{carleson1}.
The space of Carleson measures may be considered the space 
$L^\infty$ on the outer measure space, as will be discussed more
thoroughly further below.

A specific Carleson measure of interest is the following.
For some function $f\in L^\infty(\R)$ consider the 
function $F$ on $X$ defined by
$$F(y,t)=\int f(z)t^{-1} \phi(t^{-1}(y-z))\, dz\ ,$$
where $\phi$ is some smooth and rapidly decaying function 
of integral zero. Then  $|F(y,t)|^2 dy \frac{dt}t$ turns out to be
a Carleson measure\footnote{For details see the special case $p=\infty$ of \eqref{sfestimate}}. The quadratic nature of this example
suggests to define a size
$$S(F)(T(x,s))=\left( s^{-1}\int _{T(x,s)}|F(y,t)|^2 \, dy\, \frac{dt}t\right)^{1/2}\ .$$
This example provides evidence why we do not try to base
a theory of outer measure on linear averaging as could have been
done in the example of martingales or the linear averaging over 
balls.

In Example 4, the most commonly (implicitly) used size is
$$S(f)(E)=\sup_{x\in E}|f(x)|\ .$$
Rather than the $L^1$ or $L^2$ based averages from the previous examples, this is an
$L^\infty$ based average. Such an $L^\infty$ average has the effect that the more generally
defined outer $\L^p$ norms we will introduce specialize to the case of the integral (\ref{choquet}), 
which is frequently
referred to as the Choquet integral in the context of capacity theory. We
conclude this very brief discussion of Example 4 with the remark that it may be 
interesting to compare the capacitary strong type inequalites \cite{adams}, whose intensive study
goes back to the work of Maz'ya, with the embedding theorems that we discuss further below.

\subsection{A note on subadditivity}

We have chosen to only demand quasi subadditivity in the definition 
of size. Many sizes will be subadditive, which means that the constant in 
(\ref{quasi.sub}) can be chosen to be $1$. The general constant
in (\ref{quasi.sub}) allows for certain more general examples
, for example $L^p$ type sizes with $p<1$. It also sets the 
stage for quasi-subadditivity throughout our discussion, which will 
simplify some of the arguments.
 
Note that $L^p$ type sizes occur naturally in factorizations.
Generalizing the classical factorization  $|f| =|f|^{\alpha} |f|^{1-\alpha}$  for 
a Borel measurable function $f$ and $0<\alpha<1$, one may consider  
modified sizes $S^{[\alpha]}$ defined,  for every nonnegative function $f$, by
\begin{equation}\label{e.sizefactorization}
S^{[\alpha]}(f)(E) := \Big[S(f^{\frac 1{\alpha}})(E)\Big]^\alpha \ \ .
\end{equation}
One then has the factorization
$$S(f) = S^{[\alpha]}(f^{\alpha})  \ \times \ S^{[1-\alpha]}(f^{1-\alpha}) \ \ .$$
Even if $S$ is subadditive, the fractional size $S^{[\alpha]}$ with $0<\alpha<1$ might only be quasi-subadditive.

\subsection{Essential supremum and super level measure}

This section contains the most subtle points in the development
of our $L^p$ theory on outer measure spaces, with definitions carefully 
adjusted to the precise setup and the applications we have in mind.
To develop an $L^p$ theory we need a space $X$, which we assume to be a 
metric space. We need a pre-measure $\sigma$ on a collection 
$\E$ of Borel subsets, generating an outer measure $\mu$ on $X$. Finally, we need
a size $S$. So as to not overburden the notation, we collect this
data into a triple $(X,\sigma,S)$, because $\sigma$ determines 
the generating collection and the outer measure. We use the letters
$\E$ and $\mu$ for these as standing convention. We call the
triple $(X,\sigma,S)$ an outer measure space.

\begin{definition}[Outer essential supremum] \label{essential supremum}
Assume $(X,\sigma,S)$ is an outer measure space.
Given a Borel subset $F$ of $X$,
we define the outer essential supremum of $f\in \B(X)$ on $F$ to be
$$\essup_F S(f):= \sup_{E\in \E} S(f1_F)(E)\ .$$
\end{definition}

We emphasize that the values $S(f)(E)$ for fixed $f$ and all $E\in \E$ are 
in general not enough information to determine the essential supremum of 
$f$ on a 
Borel set $F$ other than $X$ or $\emptyset$. 
It is important to refer back to the
function $f$ and truncate it according to the set $F$.

We also emphasize that, unlike in Examples 1 and 2, the outer 
essential supremum in general does not coincide with the essential 
supremum of $f$ on $F$ in the Borel sense. In Example 3 with size
given by \eqref{primitivetentsize}, we note that every Lebesgue 
integrable Borel function supported above a line $t=t_0>0$ 
in the space $X$ has finite outer essential supremum. Namely
the size of such a function with respect to some tent vanishes
if the tent is small and is bounded above by $t_0^{-2}$ times the 
Lebesgue integral of the function for arbitrary tent.
On the other hand, if we define the size $S(f)(E)$ to be the
supremum of $f$ on the set $E$, then
the outer essential supremum defined above coincides with the 
classical supremum on the set $F$, under the mild assumption 
that $F$ can be covered by generating sets $E$.

The following properties of the outer essential supremum 
are inherited from the corresponding properties for the size.
We have for every $f,g\in \B(X)$ and every Borel set $F\subset X$
\begin{enumerate}
\item Monotonicity: if $|f|\le |g|$, then
$\essup_FS(f)\le  \essup_FS(g)$. 
\item Scaling: for $\lambda\in \C$ we have 
$\essup_FS(\lambda f)=  |\lambda| \essup_FS(g)$ .
\item Quasi-subadditivity: for some constant $C<\infty$
independent of $f$, $g$, $F$, 
we have $$ \essup_FS(f+g) 
\le C (\essup_FS(f)+ \essup_FS(g))\ .$$
\end{enumerate}

The use of the outer essential supremum is the main subtle point 
in the following definition.

\begin{definition}[Super level measure]
\label{superleveldefinition}
Let $(X,\sigma, S)$ be an outer measure space.
Let $f\in \B(X)$ and $\lambda>0$. We define
\begin{equation}\label{largerlambdaset}
\mu(S(f)>\lambda)
\end {equation}
to be the infimum of all values $\mu(F)$, where $F$ runs through
all Borel subset of $X$ which satisfy 
$$\essup_{X\setminus F}S(f) \le \lambda\ .$$
\end{definition}

We emphasize once more that in general $\mu(S(f)>\lambda)$ is not
the outer measure of the Borel set where $|f|$ is larger 
than $\lambda$, even though it is precisely that in many special 
examples such as the case of Lebesgue outer measure or 
in cases where the outer essential supremum above
coincides with the classical supremum.

We obtain the following properties of super level measure.
\begin{enumerate}
\item Monotonicity: if $|f|\le |g|$, then
$$\mu(S(f)>\lambda)\le \mu(S(g)>\lambda)\ .$$
\item Scaling: for a complex number $\lambda'$ we have 
$$\mu(S(\lambda' f)>|\lambda'| \lambda) = \mu(S(f)>\lambda)\ .$$
\item Quasi-subadditivity: for some constant $C<\infty$
independent of $f$, $g$, $F$, 
$$\mu(S(f+g)>C\lambda)\le \mu(S(f)>\lambda)+\mu(S(g)>\lambda)\ .$$
\end{enumerate}

Note that a constant $C=2$ would be necessary in general in the last 
inequality even if $S$ was sub-additive.

\section{Outer $\L^p$ spaces}\label{lptheory}

The definition of outer $L^p$ space and subsequent development
of the theory of outer $L^p$ spaces follows classical
lines of reasoning, once the crucial definitions of the outer 
essential supremum and the super level measure from the previous 
section have replaced their classical counterparts. 
The only minor deviation comes in the proof of the triangle 
inequality, since we do not have a satisfactory 
theory of duality in outer $L^p$ spaces. This manifests
itself in a loss of a factor $2$ in the triangle inequality.

\begin{definition}[Outer $L^\infty$]\label{d.linfty} 
Let $(X,\sigma, S)$ be an outer measure space.
Let $f\in \B(X)$, then we define
$$\|f\|_{L^\infty(X,\sigma,S)}
:=\essup_X S(f)=\sup_{E\in \E} S(f)(E) $$
and $L^\infty(X,\sigma,S)$ to be the space
of elements $f\in \B(X)$ for which $\sup_{E\in \E} S(f)(E)$ is finite.
For notational convenience we define 
$$L^{\infty,\infty}(X,\sigma,S):=L^\infty(X,\sigma,S)\ .$$
\end{definition}

As the Example 3 of Carleson measures shows, $f\in L^\infty(X,\sigma,S)$ 
need not be an essentially bounded function on $X$ in the Borel sense.

\begin{definition}[Outer $L^p$]\label{d.levelset} 
Let $0< p<\infty$.
Let $(X,\sigma, S)$ be an outer measure space.
We define for $f\in \B(X)$ :
$$\|f\|_{L^p(X,\sigma,S)}
:=\left(\int_0^\infty p\lambda^{p-1} \mu(S(f)>\lambda) \, d\lambda\right)^{1/p}\ ,$$
$$\|f\|_{L^{p,\infty}(X,\sigma,S)}:=\left(\sup_{ \lambda >0}  \lambda^{p} \mu(S(f)>\lambda)\right)^{1/p}\ .  $$
Moreover we define $L^p(X,\sigma,S)$ and 
$L^{p,\infty}(X,\sigma,S)$ to be the spaces of elements in $\B(X)$
such that the respective quantities are finite.
\end{definition}
Clearly $\mu(S(f)>\lambda)$ is monotone in $\lambda$, so that the integral
in the definition of $\|f\|_{L^p(X,\sigma,S)}$ is well defined and a number
in $[0,\infty]$. As in the classical case we trivially have
$$\|f\|_{L^{p,\infty}(X,\sigma,S)}\le \|f\|_{L^p(X,\sigma,S)}\ .$$

The following properties hold, with elementary proofs 
that follow in most cases from
the corresponding statements for super level measure.
\begin{proposition}[Basic properties of outer $L^p$]
Let $(X,\sigma,S)$ be an outer measure space and let
$f,g$ be in $\B(X)$. Then we have for
$0<p\le \infty$
\begin{enumerate}
\item Monotonicity: 
If $|f|\le |g|$, then 
$\|f\|_{L^p(X,\sigma,S)} \le \|g\|_{L^p(X,\sigma,S)}$.
\item Scaling:
$\|\lambda f\|_{L^p(X,\sigma,S)}=|\lambda| \|f\|_{L^p(X,\sigma,S)}$
for any $\lambda\in \C$.
\item Quasi-subadditivity: there is a constant $C$ independent of $f,g$ such that
$$\| f+g\|_{L^p(X,\sigma,S)}\le  C(
\|f\|_{L^p(X,\sigma,S)}+\|g\|_{L^p(X,\sigma,S)})\ .
$$
\end{enumerate}
Moreover we have for $\lambda>0$ 
$$\|f\|_{L^p(X,\lambda \sigma,S)} = \lambda^{1/p}\|f\|_{L^p(X,\sigma,S)}\ .$$
Corresponding statements hold for the spaces 
$L^{p,\infty}(X,\sigma,S)$.

\end{proposition}

Note that the proof of quasi-subadditivity for $L^p$ with $p<\infty$
is based on quasi-subadditivity of super level measure, which 
yields a constant $C$ different from $1$ even if the size $S$ is subadditive.
It might be interesting to study conditions under which one may have 
subadditivity for $L^p$.

We turn to the behaviour of outer $L^p$ spaces under mappings 
between outer measure spaces.
Note that Borel measurable functions as well as classical $L^p$ 
functions are typically pulled back under a continuous map, while 
in contrast Borel measures are pushed forward under such maps.
This is one of the motivations for us to use the class of Borel
measurable functions to develop the theory of outer $L^p$ functions,
even though much of the theory can be developed for Borel measures
as well.

Let $X_1$ and $X_2$ be two metric spaces and let $\Phi:X_1\to X_2$ be
a continuous map. For $j=1,2$ let $\E_j$ be a collection of Borel 
sets covering $X_j$ 
and let $\sigma_j:\E_j\to [0,\infty]$ be a function generating an
outer measure $\mu_j$ on $X_j$. 
Let $S_1$ and $S_2$ be sizes turning $(X_1,\sigma_1,S_1)$
and $(X_2,\sigma_2,S_2)$ into outer measure spaces.

\begin{proposition}[Pull back]\label{pullback}
Assume that for every $E_2\in \E_2$ we have
\begin{equation}\label{mupullback}
\mu_1(\Phi^{-1} E_2)\le A \mu_2(E_2)\ .
\end{equation}
Further assume that for each $E_1\in \E_1$ there exists 
$E_2\in \E_2$  such that for every $f\in \B(X_2)$ we have
\begin{equation}\label{spullback}
S_1(f\circ\Phi)(E_1)\le B S_2(f)(E_2)\ .
\end{equation}
Then we have for every $f\in \B(X_2)$ and $0<p  \le    \infty$
and some universal constant $C$:
$$\|f\circ \Phi\|_{L^p(X_1,\sigma_1,S_1)}\le A^{1/p}B C \|f\|_{L^p(X_2,\sigma_2,S_2)}\ ,$$
$$\|f\circ \Phi\|_{L^{p,\infty}(X_1,\sigma_1,S_1)}\le A^{1/p}B C \|f\|_{L^{p,\infty}(X_2,\sigma_2,S_2)}\ .$$
\end{proposition}

\begin{proof}
First note that by scaling properties it is no restriction to prove
the proposition with constants $A=B=1$ in \eqref{mupullback}
and \eqref{spullback}.

For every Borel set $F_2\subset X_2$ we have
$$\mu_1(\Phi^{-1} F_2)\le \mu_2(F_2)\ .$$ 
Namely, given $F_2$  without loss of generality we may assume that $\mu_2(F_2)<\infty$. Let $\E'_2\subset \E_2$ be a cover of $F_2$
which attains, up to a factor $(1+\epsilon)$ with small $\epsilon>0$, 
the outer measure of $F_2$: 
$$\sum_{E_2\in \E'_2} \mu_2(E_2)\le \sum_{E_2\in \E'_2} \sigma_2(E_2)\le  (1+\epsilon)\mu_2(F_2)\ .$$
Then we obtain
$$\mu_1(\Phi^{-1}(F_2))\le \sum_{E_2\in \E'_2} \mu_1(\Phi^{-1}(E_2))
\le \sum_{E_2\in \E'_2}\mu_2(E_2)\le (1+\epsilon) \mu_2(F_2)\ .$$
This proves the claim, since $\epsilon>0$ can be chosen arbitrarily.

Assume $F\subset X_2$ is a Borel set such that 
$$
 \mu_2(F)\le 
(1+\epsilon)\mu_2(S_2(f) >\lambda)
$$
and for every $E_2\in \E_2$ we have $S_2(f1_{F^c})(E_2)\le \lambda$.
Pick $E_1\in \E_1$, then there exists  $E_2 \in \E_2$  such that we have
$$
S_1((f\circ \Phi)(1_{\Phi^{-1}(F^c)}))(E_1)=
S_1((f1_{F^c})\circ \Phi)(E_1)$$
$$\le S_2(f1_{F^c})(E_2)\le \lambda\ ,$$
and 
hence
$$\mu_1(S_1(f\circ \Phi)\ge \lambda)\le 
\mu_1((\Phi^{-1}(F^c))^c)$$
$$\le \mu_1(\Phi^{-1}F)\le \mu_2(F)
\le  (1+\epsilon) \mu_2(S_2(f)>\lambda)\ .$$
This proves the desired inequalities for $p<\infty$.
The case $p=\infty$ follows immediately from the assumption
on sizes.
\end{proof}

\begin{proposition}[Logarithmic convexity] \label{p.convex-interpolation}
Let $(X,\sigma,S)$ be an outer measure space and let
$f\in \B(X)$.
Assume $\alpha_1+\alpha_2=1, \ \ 0< \alpha_1,\alpha_2< 1$, and
$$1/p = \alpha_1 / p_1  + {\alpha_2} / {p_2}$$
for $p_1,p_2\in (0,\infty]$ with $p_1\ne p_2$. Then
\begin{equation}\label{convexityc}
\|f\|_{L^{p}(X,\sigma,S)} 
\le C_{p,p_1,p_2}\Big(\|f\|_{L^{p_1,\infty}(X,\sigma,S)} 
\Big)^{\alpha_1} \Big(\|f\|_{L^{p_2,\infty}(X,\sigma,S)} \Big)^{\alpha_2} \ \ . 
\end{equation}
\end{proposition}
\begin{proof} 
Assume without loss of generality 
$p_1<p_2$. We first consider the case $p_2<\infty$.
If either of the norms on the right-hand-side of (\ref{convexityc}) vanishes,
then $\mu(S(f) > \lambda)$ vanishes for all $\lambda>0$ and then the 
left-hand-side of (\ref{convexityc}) vanishes as well. 
By scaling we may then assume 
$$A:=\|f\|_{L^{p_1,\infty}(X,\sigma,S)}^{p_1}=\|f\|_{L^{p_2,\infty}(X,\sigma,S)}^{p_2}\ .$$ 
Optimizing the use of these two identites we have with $p_1<p<p_2$
$$\mu(S(f)>\lambda) \le A \min(\lambda^{-p_2}, \lambda^{-p_1}) \ \ ,$$
$$\|f\|_p \le (Ap(\int_0^{1} \lambda^{p-p_1-1} d\lambda +  \int_{1}^\infty \lambda^{p-p_2-1} d\lambda))^{\frac 1 {p}} $$
$$\le C_{p,p_1,p_2} A^{1/p}=
 C_{p,p_1,p_2}\Big(\|f\|_{L^{p_1,\infty}(X,\sigma,S,)}^{p_1} 
\Big)^{\frac{\alpha_1}{p_1}} \Big(\|f\|_{L^{p_2,\infty}(X,\sigma,S)}^{p_2} 
\Big)^{\frac{\alpha_2}{p_2}} \ \ . $$
This completes the proof in case $p_2<\infty$. 
If $p_2 = \infty$, we may assume by scaling that 
$\|f\|_{L^{\infty}(X,\sigma,S)}=1$.
Then for $\lambda > 1$ we have $\mu(S(f)>\lambda) = 0 $. 
Consequently, 
$$\|f\|_{L^p(X,\sigma,S)} 
\le  (p \|f\|_{L^{p_1,\infty}(X,\sigma,S)}^{p_1} 
\int_0^{1} \lambda^{p-p_1-1} d\lambda)^{\frac 1 {p}}  \le   C_{p,p_1,p_2} 
\|f\|_{L^{p_1,\infty}(X,\sigma,S)}^{\alpha_1} \ \ .$$
\end{proof}

\begin{proposition}[H\"older's inequality]\label{p.hoelder-energy}

Assume we have a metric space $X$, three collections $\E,\E_1,\E_2$
of Borel subsets, three functions $\sigma,\sigma_1,\sigma_2$ on these
collections generating outer measures $\mu,\mu_1,\mu_2$ on $X$. 
Assume $\mu\le \mu_j$ for $j=1,2$.
Assume $S,S_1,S_2$ are three respective sizes such that
for any $E\in \E$ there exist $E_1\in \E_1$ and $E_2\in \E_2$ 
such that for all $f_1,f_2\in \B(X)$ we have
\begin{equation}\label{e.size-factorization}
S(f_1f_2)(E)\le S_1(f_1)(E_1) S_2(f_2)(E_2) \ \ .
\end{equation}
Let $p, p_1,p_2\in (0,\infty]$ such that $1/p = 1/p_1 + 1/p_2$. Then
\begin{eqnarray}
\label{e.hoelder-strong}
\|f_1f_2\|_{L^p(X,\sigma,S)} \le
2 \|f_1\|_{L^{p_1}(X,\sigma_1,S_1)}\|f_2\|_{L^{p_2}(X,\sigma_2,S_2)} \ \ .
\end{eqnarray}
\end{proposition}
\begin{proof} We assume $0< p_1,p_2<\infty$, the case 
$\max(p_1,p_2)= \infty$ can be argued similarly. Without loss of generality assume that the factors on the right hand side of \eqref{e.hoelder-strong} are  finite. 
For $j=1,2$ pick Borel sets $F_j\subset X$ such that
for every $E_j\in \E_j$ we have
$$S_j(f_j 1_{F_j^c})(E_j)\le \lambda^{p/p_j}$$
and 
$$\mu_j(F_j)\le \mu_j(S_j(f_j)>\lambda^{p/p_j})+\epsilon\ .$$
Define $F=F_1\cup F_2$.
Let $E\in \E$ be arbitrary, then by 
\eqref{e.size-factorization}
there exists  $E_1 \in \E_1$ and $E_2\in \E_2$  such that
$$
S(f_1f_2 1_{F^c})(E)\le 
S_1(f_1 1_{F^c})(E_1) S_2(f_2 1_{F^c})(E_2)$$
$$\le  S_1(f_1 1_{F_1^c})(E_1) S_2(f_2 1_{F_2^c})(E_2)\le \lambda^{p/p_1}\lambda^{p/p_2}=\lambda\ ,$$
the passage from the first to second line by monotonicity of the sizes.

It follows from 
 subadditivity of $\mu$ and domination of $\mu$ by $\mu_1$ and $\mu_2$
that for all $\lambda>0$ 
\begin{equation}\label{e.Clambda-break}
\mu(S(f_1f_2)>\lambda) \le \mu(F)\le \mu(F_1)+\mu(F_2)$$
$$\le \mu_1(F_1)+\mu_2(F_2)\le 
2\epsilon+ \sum_{i=1}^2 \mu(S_i(f_i) > \lambda^{p/p_i}) \ \ .
\end{equation}
To prove \eqref{e.hoelder-strong}
we may assume via scaling that
$$\|f_1\|_{L^{p_1}(X,\sigma_1,S_1)} = \|f_2\|_{L^{p_2}(X,\sigma_2,S_2)} =1 \ \ .$$
Then \eqref{e.hoelder-strong} follows from \eqref{e.Clambda-break},
using that $\epsilon>0$ is arbitrarily small,
$$\int p\lambda^{p-1} \mu(S(f_1f_2)>\lambda) d\lambda \le 
\int p\lambda^{p-1} \sum_{i=1}^2 \mu(S_i(f_i) > \lambda^{p/p_i}) d\lambda$$
$$ = \sum_{i=1}^2 \int p_i \lambda^{p_i-1} \mu(S_i(f_i) > \lambda ) d\lambda
 = 2 \ \ .$$
\end{proof}

In the following proposition, let $L^p(Y,\nu)$ denote the classical 
space of complex valued functions on a measure space $(Y,\nu)$ such that 
$\|f\|_{L^p(Y,\nu)}:= (\int_Y |f(x)|^p d\nu)^{1/p}$ is finite.

The following proposition is an outer measure version of classical
Marcinkiewicz interpolation, which in practice is used to obtain
strong bounds in a range of exponents $p$ from weak bounds at the
endpoints of the range.

\begin{proposition}[Marcinkiewicz interpolation]\label{p.marcinkiewicz} 
Let $(X,\sigma,S)$ be an outer measure space.
Assume $1\le p_1<p_2 \le \infty$.   Let $T$ be an operator that maps $L^{p_1}(Y,\nu)$ and $L^{p_2}(Y,\nu)$ 
to the space of Borel functions on $X$, such that for any
$f,g\in L^{p_1}(Y,\nu)+L^{p_2}(Y,\nu)$ and $\lambda\ge 0$ we have
\begin{enumerate}
\item Scaling: $|T(\lambda f)|=|\lambda T(f)|$.
\item Quasi subadditivity: $|T(f+g)|\le C(|T(f)|+|T(g)|)$.
\item Boundedness properties:
$$ \|T(f)\|_{L^{p_1,\infty}(X,\sigma,S)}\le A_1\|f\|_{L^{p_1}(Y,\nu)}\ ,$$
$$ \|T(f)\|_{L^{p_2,\infty}(X,\sigma,S)}\le A_2\|f\|_{L^{p_2}(Y,\nu)}\ .$$
\end{enumerate}
Then we also have
$$ \|T(f)\|_{L^{p}(X,\sigma,S)}\le A_1^{\theta_1}A_2^{\theta_2}C_{p_1,p_2,p}\|f\|_{L^{p}(Y,\nu)}\ ,$$
where $p_1<p<p_2$ and $\theta_1$, $\theta_2$ are such that
$$\theta_1+\theta_2=1\ ,$$
$$ \frac 1p = \frac {\theta_1} {p_1}+\frac {\theta_2}{p_2}\ .$$
\end{proposition}

\begin{proof} 
We may normalize $\nu$ to become $\tilde{\nu}=\lambda^{-1}\nu$, with $\lambda$ chosen so that
$$A_1 \lambda^{1/p_1}=A_2\lambda ^{1/p_2}:=A\ .$$
Then
$$A_1^{\theta_1}A_2^{\theta_2}\lambda^{1/p}=A \lambda^{-\theta_1/p_1-\theta_2/p_2} \lambda^{1/p}=A\ .$$
Thus it suffices to prove the theorem with $A_1=A_2=A$.
Further normalizing $T$ to become $\tilde{T}=A^{-1}T$, we observe 
that it suffices to prove the theorem with $A_1=A_2=1$.

If $f_1\in L^{p_1}(Y,\nu)$ and $f_2\in L^{p_2}(Y,\nu)$, then
we have
for every $E\in \E$
$$S(T(f_1+f_2))(E)  \le C\Big( S(Tf_1)(E) +  S(Tf_2)(E)) \Big)\ .$$
Then we also have for some possibly different constant $C$:
\begin{equation}\label{e.muTsublinear}
\mu(S(T(f_1+f_2))> C\lambda) \le \mu(S(Tf_1)>\lambda) + \mu(S(Tf_2)>\lambda) \ .
\end{equation}
We first assume: $0<p_1<p_2<\infty$.  Let $f\in L^p(Y,\nu)$.  We decompose $f= f_{1,\lambda} + f_{2,\lambda}$ with $f_{1,\lambda}=f 1_{|f|>\lambda}$. It is clear that $f_{j,\lambda}\in L^{p_j}(Y,\nu)$. 
Using \eqref{e.muTsublinear} we obtain
$$\mu(S(Tf)>C\lambda) \le C \sum_{j=1}^2 \lambda^{-p_j}\|f_{j,\lambda}\|_{p_j}^{p_j}$$
$$=C 
\lambda^{-p_1} \int_Y |f|^{p_1} 1_{|f|> \lambda}\, d\nu(y)+
C \lambda^{-p_2} \int_Y |f|^{p_2} 1_{|f|\le \lambda}\, d\nu(y)\ ,$$
and therefore
$$\|Tf\|_{L^p(X,\sigma,S)}=\left(p \int_0^\infty \lambda^{p-1} \mu(S(Tf)>\lambda)\, d\lambda\right)^{1/p}
$$
$$  \le
 C\Big(  \int_Y |f|^{p_1} (\int_0^{|f|}  \lambda^{p-p_1-1}   d\lambda) d\nu +   \int_Y  |f|^{p_2}  (\int_{|f|}^\infty \lambda^{p-p_2-1}   d\lambda )d\nu\Big)^{1/p} \ \ ,$$
$$\le  C \|f\|_{L^p(Y,\nu)} \ \ .$$
It remains to consider the case $p_1<p_2=\infty$. We similarly decompose $f= f_{1,\lambda} + f_{2,\lambda}$ with $f_{1,\lambda}=f 1_{|f|>c\lambda}$ for suitable small $c$
to be determined momentarily.   Then  
$$\| Tf_{2,\lambda}\|_{L^\infty(X,\sigma,S)} \le \|f_{2,\lambda}\|_{L^\infty(Y,\nu)} < c\lambda \ \ .$$
It follows 
from \eqref{e.muTsublinear} that with sufficiently small $c$
$$\mu(S(Tf)> \lambda) \le \mu(S(Tf_{1,\lambda})> \lambda/C) + \mu(S(Tf_{2,\lambda})> \lambda/C) 
= \mu(S(Tf_{1,\lambda})> \lambda/C)  \ \ .$$
Consequently,
$$\|Tf\|_{L^p(X,\sigma,S)} \le C \Big(\int_0^\infty \lambda^{p-1} \mu(S(Tf)>\lambda)d\lambda\Big)^{1/p}$$
$$\le C  \Big(\int_0^\infty \lambda^{p-1} \mu(S(Tf_{1,\lambda})>\lambda/C)d\lambda\Big)^{1/p}\ .$$

Then we proceed as before to obtain
$$\|Tf\|_{L^p(X,\sigma,S)} \le  
C \int_Y  |f|^{p_1}  (\int_{|f|}^\infty\lambda^{p-p_1-1}   d\lambda )d\nu\Big)^{1/p} \le C \|f\|_{L^p(Y,\nu)}\ \ .$$
\end{proof}

The following is a simple variant of a
classical fact about measures: If a measure
$\nu$ on a space is absolutely continuous
with respect to another measure $\mu$, and if the Radon
Nikodym derivative of $\nu$ with respect to
$\mu$ is bounded, then the total mass of $\nu$
can be estimated by the total mass of $\mu$. 

\begin{proposition}\label{measuredomination} 
Assume $(X,\sigma,S)$ is an outer measure space and
assume that about every point in $X$ there is
an open ball for which there exists
$E\in \E$ which contains the ball.
Let $\nu$ be a positive Borel measure on $X$.
Assume that for every $f\in \B(X)$ and
for every $E\in \E$ we have
$${\int_E |f| \, d\nu} \le C S(f)(E) \sigma(E)\ .$$

Then, for every $f\in \B(X)$ with finite $\|f\|_{L^\infty(X,\sigma,S)}$
we have:
$$|\int_X f\, d\nu |\le C \|f\|_{L^1(X,\sigma,S)}\ ,$$
where the implicit constant $C$ in particular is independent of $\|f\|_{L^\infty(X,\sigma,S)}$.
\end{proposition}

\begin{proof}
We may assume that $\mu(S(f)>\lambda)$ is finite for every
$\lambda>0$, or else nothing is to prove.
For each $k\in \Z$ consider a set $F_k$
such that
$$\essup_{F_k^c}S(f) \le 2^k\ ,$$
$$\mu(F_k)\le 2\mu(S(f)>2^k)\ .$$
Cover $F_k$ by a countable subcollection $\E_k$
of $\E$ such that
$$\sum_{E\in \E_k}{\sigma(E)}\le 2\mu(F_k)\ .$$

Let $F=\bigcup_k F_k$ and note that for every 
sufficiently small open ball $B$ about a point in $X$
we can find $E\in \E$ such that  $B  \subset E$, thus 
$$\int_{B\cap F^c} |f|1_{F^c}\, d\nu\le C S(f1_{F^c})(E) \sigma(E) =0\ .$$
Hence
$$\int_X|f|\, d\nu=\int_F |f|\, d\nu\ .$$
Since we may assume $F_k=\emptyset$ for sufficiently large $k$ we have 
$$|\int_X f\, d\nu |\le \sum_k \int_{F_k\setminus \bigcup_{l>k} F_l} |f|\, d\nu
\le\sum_k \sum_{E\in \E_k} \int_{E\setminus \bigcup_{l>k} F_l} |f|\, d\nu$$
$$\le \sum_k \sum_{E\in \E_k} S(f1_{F_{k+1}^c}) \mu(E)\le 
C\sum_k \sum_{E\in \E_k} 2^k  \sigma(E)$$
$$\le C\sum_k 2^k \mu(S(f)>2^k)\le C\|f\|_{L^1(X,\sigma,S)}\ .$$
This completes the proof of the proposition.
\end{proof}

\section{
Carleson embedding, paraproducts, and the $T(1)$ theorem}
\label{t1section}

This section contains classical results rephrased in the language of 
outer measure spaces utilizing Example 3 of Section \ref{examplesection}. 
Readers interested in reviewing the 
classical theory are referred to \cite{stein}. A novelty
of our approach is the interpretation of Carleson embedding
theorems as boundedness of certain  maps from a classical $L^p$ to an outer $L^p$ space.
As a consequence, outer H\"older's inequality can be used to
prove various multi-linear estimates such as paraproduct
estimates or a core version of a $T(1)$ theorem.

\subsection{Carleson embeddings}

We consider the upper half plane $X=\R\times (0,\infty)$, we  
let $\E$ be the collection of tents 
$$T(x,s)=\{(y,t)\in \R\times (0,\infty): t<s, |x-y|<s-t\}\ ,$$
and we set $\sigma(T(x,s))=s$ as in Example 3.

\setlength{\unitlength}{0.4mm}
\begin{figure}[ht]
\caption{The tents $T(x,s)$ and $T(x',s')$.}
\begin{picture}(220,100)
\put(10,10){\vector(0,1){70}}
\put(5,10){\vector(1,0){195}}
\put(205,5){$y$}
\put(5,85){$t$}

\put(40,10){\line(1,1){20}}
\put(60,30){\line(1,-1){20}}
\put(100,10){\line(1,1){40}}
\put(140,50){\line(1,-1){40}}
\put(50,40){$(x,s)$}
\put(130,60){$(x',s')$}
\put(60,30){\circle*{3}}
\put(140,50){\circle*{3}}
\end{picture}
\end{figure}
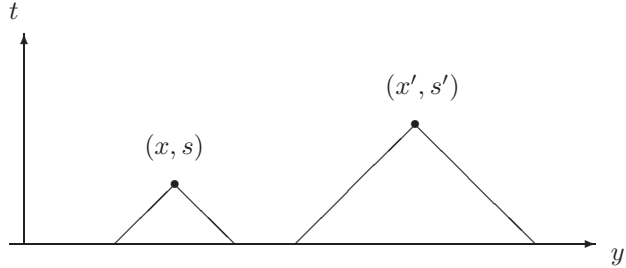

Define for $1\le p<\infty$ 
the sizes
$$S_p (F)(T(x,s)):=(s^{-1}\int_{T(x,s)} |F(y,t)|^p \, dy\, \frac{dt}t)^{1/p}\ \ ,$$
where we have used standard Lebesgue integration in $\R\times (0,\infty)$, and
$$S_\infty (F)(T(x,s)):=\sup_{(y,t)\in T(x,s)} |F(y,t)| \ \ .$$

Let $\phi$ be a smooth function on the real line 
supported in $[-1,1]$ and define for a locally integrable function $f$
on the real line
\begin{equation}\label{definefphi}
F_\phi(f)(y,t):= \int f(x) t^{-1}\overline{\phi(t^{-1}(y-x))}\, dx
\ . \end{equation}
The mapping $f\to F_\phi(f)$ is an embedding of a space
of functions on the real line into a space of functions in the upper
half plane reminiscent of {\it Carleson embeddings}. Thus we call the following estimates
{\it Carleson embedding theorems}, even though traditionally this 
notion is reserved for special instances and applications of such estimates. In particular, if $\nu$ is a Borel measure on the upper half plane satisfying the so-called Carleson measure condition $\nu(T(x,s))\le Ms$, then one could deduce from Theorem~\ref{carlesonembedding} a typical version of the classical Carleson embedding theorem, as follows. Below the first and last $L^p$ norm are classical Lebesgue norms while the second and third $L^p$ norms are outer $L^p$ norms over an outer measure generated by $\nu$ and $\sigma$ and the tent collection.
$$\|F_\phi(f)\|_{L^p(X,\nu)} \le \|F_\phi\|_{L^p(X,\nu,S_\infty)} \le $$
$$\le M\|F_\phi\|_{L^p(X,\sigma,S_\infty)} \le$$
$$\le  C_{p,\phi}M\|f\|_p$$

\begin{theorem}\label{carlesonembedding}
Let $1<p\le \infty$. We have for $\phi$ as above 
\begin{equation}\label{mfestimate}
\|F_\phi(f)\|_{L^{p}(X,\sigma,S_\infty)}\le C_{p,\phi}\|f\|_p\ .
\end{equation}
If in addition $\int \phi=0$, then
\begin{equation}\label{sfestimate}
\|F_\phi(f)\|_{L^{p}(X,\sigma,S_2)}\le C_{p,\phi}\|f\|_p\ .
\end{equation}
\end{theorem}

\begin{proof}
We first prove Estimate (\ref{mfestimate}). The estimate
will follow by Marcinkiewicz interpolation, Proposition
\ref{p.marcinkiewicz}, between weak
endpoint bounds at $p=\infty$ and $p=1$.
Clearly we have for all $(y,t)\in X$:
$$|F_\phi(f)(y,t)|\le \|f\|_\infty \|\phi\|_1\ .$$
Hence 
$$S_\infty(F_\phi)(T(x,s))\le \|f\|_\infty \|\phi\|_1$$
for every tent $T(x,s)$, and this implies the $L^\infty$ bound.
To prove the weak type estimate at $L^1$, fix $f$ and $\lambda>0$.
Consider the set $\Omega \subset \R$ where the Hardy Littlewood maximal
function $Mf$ of $f$ is larger than $c_\phi \lambda$ for some constant
$c_\phi$ that depends on $\phi$ and is specified later. The set 
$\Omega$ is open and thus the disjoint
union of at most
countably many open intervals $(x_i-s_i,x_i+s_i)$ for $i=1,2,\dots$.
Let $E$ be the union of the tents $T(x_i,s_i)$.
Then the geometry of tents implies that for $(x,s)\not\in E$ 
none of the intervals $(x_i-s_i,x_i+s_i)$ may contain the interval
$(x-s,x+s)$ and hence  there is a point $y\in (x-s,x+s)$ such that
$Mf(y)\le c_\phi\lambda$. Then we see from a standard
estimate of $\phi$ by a superposition of characteristic functions
of intervals of length at least $2s$:
$$F_\phi(f)(x,s)\le C_\phi Mf(y)\le \lambda\ ,$$
the latter by appropriate choice of $c_\phi$.
Hence 
$$\essup_{E^c} S_\infty(F_\phi)\le \lambda\ .$$
On the other hand, by the Hardy Littlewood maximal theorem, 
$$\mu(E)\le \sum_i s_i\le  |\{x: Mf(x)\ge c_\phi \lambda\}|\le C_\phi \|f\|_1 \lambda^{-1}\ .$$
This proves the weak type estimate at $L^1$ and completes the proof
of Estimate (\ref{mfestimate}).

We turn to Estimate (\ref{sfestimate}), which is proven similarly by Marcinkiewicz
interpolation between weak endpoint bounds at $\infty$ and $1$.
Note first that if $\phi$ has integral zero, then the map
$F_\phi$ goes under the name of ``continuous wavelet transform'' and
is well known to be a multiple of an isometry in the following sense:
$$\int_0^\infty  \int_\R |F_\phi(g)(y,t)|^2 \, dy \,\frac{dt}t=C_\phi \|g\|_2^2$$
for every $g\in L^2(\R)$.
This fact goes under the name of Calder\'on's reproducing formula 
or Calder\'on's resolution of the identity, see for example \cite{daubechies}. 
It can be proven by a calculation similar to our reduction of
Theorem \ref{bilinearhttheo}
to Lemma \ref{bhttheorem} below.

Consider a tent $T(x,s)$. For $(y,t)$ in the tent, we see from
compact support of $\phi$ that
$$F_\phi(f)(y,t)=F_\phi(f1_{[x-3s,x+3s]})(y,t)\ .$$
Applying Calder\'on's reproducing formula with $g=f1_{[x-3s,x+3s]}$ gives
$$\int  \int_{T(x,s)} |F_\phi(f)(y,t)|^2 \, dy \,\frac{dt}t\le C_\phi \|f1_{[x-3s,x+3s]}\|_2^2\le C_\phi s \|f\|_\infty^2 \ .$$
Dividing by $s$ gives
$$S_2(F_\phi(f))(T(x,s))\le C_\phi \|f\|_\infty\ ,$$
which proves the desired estimate for $p=\infty$.
 
To prove the weak type bound at $p=1$, fix $f\in L^1(\R)$ and $\lambda>0$ and consider again
the set $\Omega=\{x: Mf(x)>c_\phi \lambda\}$, which is the disjoint
union of open intervals $(x_i-s_i,x_i+s_i)$. Consider the
Calder\'on-Zygmund decomposition of $f$ at level $c_\phi \lambda$:
$$f=g+\sum_i b_i$$
which is uniquely determined by the demand that for each $i$ the 
function $b_i$ is supported on $[x_i-s_i,x_i+s_i]$, and has integral zero, while
$g$ is constant on this interval.
As a consequence, $g$ is bounded by $c_\phi \lambda$ and we have by the previous argument
for any tent
$$S_2(F_\phi(g))(T(x,s))\le \lambda/2\ .$$

Let $E$ be the union of tents $T(x_i,3s_i)$. Let $b=\sum_i b_i$. It remains to show that, with small choice of $c_\phi$, for every $(x,s)\in \R\times (0,\infty)$ it holds that
$$S_2(F_\phi(b)1_{E^c})(T(x,s))\le \lambda/2 \ \ .$$
Let
$B_i$ denote the compactly supported primitive of $b_i$.  Then we have for $(y,t)\notin E$,
using compact support of $\phi$, 
$$F_\phi(b)(y,t)= \int b(x) t^{-1}\overline{\phi(t^{-1}(y-x))}\, dy$$
$$=\int \sum_{i: s_i \le t} b_i(x) t^{-1}\overline{\phi(t^{-1}(y-x))}\, dx$$
$$=  \int \sum_{i: s_i \le t} B_i(x) t^{-2}\overline{\phi'(t^{-1}(y-x))}\, dx\ .$$
Hence
$$|F_\phi(b)(y,t)|\le \|\sum_{i: s_i \le t} t^{-1} B_i \|_\infty \|\phi'\|_1\ .$$
We claim that the $L^\infty$ norm on the right-hand-side is bounded by $4c_\phi\lambda$.
Since the $B_i$ are disjointly supported, it suffices to see 
$\|t^{-1}B_i\|_\infty\le 4c_\phi \lambda$ for each $i$ with $s_i\le t$. However, this
follows from $\|b_i\|_1\le 4c_\phi \lambda s_i$, which is a standard estimate
for the Calder\'on Zygmund decomposition. Hence
$$S_\infty(F_\phi(b)1_{E^c})(T(s,x))\le 4c_\phi \lambda \  .$$
To obtain a bound for $S_2$ in place of $S_\infty$, we use
log convexity of $S_p$ and a bound on $S_1$.
Let $T(x,s)$ be a tent and $b_i$ one summand of the bad function. Then we have
from considerations of the support of $b_i$ and $\phi$:
$$\int_{(y,t)\in T(x,s)\setminus E} |\int_\R b_i(z) t^{-1}\overline{\phi(t^{-1}(y-z))}\, dz|\, dy\, \frac{dt}t$$
$$\le \int_{t \ge s_i} \int _{|y-x_i| \le 2t} |\int_\R b_i(z) t^{-1}\overline{\phi(t^{-1}(y-z))}\, dz|\, dy\, \frac{dt}t\ .$$
Using partial integration we estimate this by
$$\int_{t \ge s_i} \int _{|y-x_i| \le 2t} \|B_i\|_1 \|\phi'\|_\infty \, dy\, \frac{dt}{t^3}$$
$$\le C_\phi \int_{t \ge s_i}  \|B_i\|_1 \, \frac{dt}{t^2}
\le C_\phi \|B_i\|_1 s_i^{-1}\le C_\phi \|b_i\|_1 \le  \lambda s_i/6\ .$$

Adding over the disjointly supported $b_i$ inside $(x-3s,x+3s)$, which are
all the summands of the bad function possibly contributing to $F_\phi(b)$ on $T(x,s)$, gives
$$S_1(F_\phi(b)1_{E^c})(T(x,s)) \le \lambda /2\ .$$
By log convexity, we then obtain
$$S_2(F_\phi(b)1_{E^c})(T(x,s)) \le \lambda/2 \ .$$
Together with the previously established bound for the 
good function we obtain by the triangle inequality 
$$ S_2(F_\phi(f)1_{E^c})(T(x,s)) \le \lambda$$
and hence
$$\essup_{E^c} S_2(F_\phi(f))\le \lambda\ .$$
On the other hand, we have by the Hardy Littlewood maximal theorem as before
$$\mu(E)\le C_\phi \|f\|_1 \lambda^{-1}\ .$$
This completes the proof of the weak type $1$ endpoint bound for Estimate \ref{sfestimate} and thus the proof of Theorem \ref{carlesonembedding}.
\end{proof}

We will need to apply Theorem \ref{carlesonembedding} in
a slightly modified setting. 

For two parameters $-1\le \alpha\le 1$ and $0< \beta\le 1$
define
$$F_{\alpha,\beta,\phi}(f)(y,t):=F_\phi(f)(y+\alpha t, \beta t)\ .$$
To estimate the outer $L^p$ norm of $F_{\alpha,\beta,\phi}(f)$,
first note that by a simple change of variables
$$s^{-1}\int_{T(x,s)}|F_{\alpha,\beta,\phi}(f)(y,t)|^2\, dy \frac{dt}t
=s^{-1}\int_{T_{\alpha,\beta}(x,s)}|F_\phi(f)(y,t)|^2\, dy \frac{dt}t$$
where we have defined the modified tent $T_{\alpha,\beta}(x,s)$ to be
the set of all points $(z,u)$ such that $(z-\alpha \beta^{-1} u,\beta^{-1} u)\in T(x,s)$.
This modified tent is a tilted triangle, it has height 
$\beta s$ above the real line and width $2s$ near the real line.
The tip of the tilted tent is the point $(x+\alpha s,\beta s)$, which 
is contained in a rectangle with base $[x-s,x+s]$ and height $s$
above the $x$-axis. 
We construct an outer measure space using the collection of modified
tents by setting
$$\sigma_{\alpha,\beta}(T_{\alpha,\beta}(x,s)):=s\ .$$
We then define for a Borel measurable function $G$ on $X$ 
$$S_{\alpha,\beta,2} (G)(T_{\alpha,\beta}(x,s)):=(s^{-1}\int_{T_{\alpha,\beta}(x,s)} |G(y,t)|^2 \, dy\, \frac{dt}t)^{1/2}\ .$$
We have by transport of structure
$$\|F_{\alpha,\beta,\phi}(f)\|_{L^p(X,\sigma,S_2)}=
\|F_\phi(f)\|_{L^p(X,\sigma_{\alpha,\beta},S_{\alpha,\beta,2})}\ .$$
Given a standard tent $T(x,s)$, we may cover it by a modified tent
$T_{\alpha,\beta}(x',s')$ of width $ 2s'=4\beta^{-1}s$. Hence
$$\mu_{\alpha,\beta}(T(x,s))\le C \beta^{-1} \mu(T(x,s))\ .$$
Moreover, a modified tent $T_{\alpha,\beta}(x,s)$ is contained in a standard
tent $T(x',s')$ of width $2s'=4s$. Hence
$$S_{\alpha,\beta,2}(G)(T_{\alpha,\beta}(x,s))\le C S_2(G)(T(x',s'))\ .$$
Thus Proposition \ref{pullback} applied to the identity map on $X$ gives
$$\|F_\phi(f)\|_{L^p(X,\sigma_{\alpha,\beta}, S_{\alpha,\beta,2})}\le C\beta^{-1/p}
\|F_\phi(f)\|_{L^p(X,\sigma, S_2)}\ .$$
We have thus proven the following corollary.
\begin{corollary}\label{tentcorollary}
Assume the setup as above. 
Let $1<p\le \infty$ and $-1\le \alpha\le 1$ and $0< \beta\le 1$
and assume $\int \phi=0$. Then
$$
\|F_{\alpha,\beta,\phi}(f)\|_{L^{p}(X,\sigma,S_2)}\le C_{p,\phi} \beta^{-1/p}\|f\|_p\ .
$$
\end{corollary}
We shall need a slightly better dependence on the parameter $\beta$
in the last corollary. This is stated in the following lemma, where
explicit values for $\epsilon$ are not difficult to obtain but
unimportant for our purpose.

\begin{lemma}\label{tiltedembedding}
Assume the setup as above. 
Let $1<p\le \infty$ and $-1\le \alpha\le 1$ and $0< \beta\le 1$
and assume $\int \phi=0$. Then there exists $\epsilon>0$ such that we have 
$$
\|F_{\alpha,\beta,\phi}(f)\|_{L^{p}(X,\sigma,S_2)}\le C_{p,\phi} \beta^{-1/p+\epsilon}\|f\|_p\ .
$$
\end{lemma}

Proof:
This lemma follows by various applications of Marcinkiewicz 
interpolation using the bounds of Corollary \ref{tentcorollary}
and an improved weak type 2 bound:
$$
\|F_{\alpha,\beta,\phi}(f)\|_{L^{2,\infty}(X,\sigma,S_2)}=
\|F_{\phi}(f)\|_{L^{2,\infty}(X,\sigma_{\alpha,\beta},S_{\alpha,\beta,2})}
\le C_{\phi} \|f\|_2\ ,
$$
where the right-hand-side does not depend on $\beta$.
To see this bound, fix $f\in L^2(\R)$ and $\lambda>0$. 
Consider the collection $\I$ of all open intervals $(x-s,x+s)$
on the real line such that
$$
s^{-1}
\int_{T_{\alpha,\beta}(x,s)}|F_\phi(f)(y , t)|^2\, dy \frac{dt}t
> \lambda^2\ .$$
The union $\bigcup_{I\in \I}I$ is an open set which can be written
as the disjoint union of  countably many  open intervals $(x_i-s_i,x_i+s_i)$.
If we set $E=\bigcup_i T_{\alpha,\beta}(x_i,s_i)$, then it is clear
that
$$S_{\alpha,\beta,2}(F_{\phi}(f) 1_{E^c})(T_{\alpha,\beta}(x,s)) \le \lambda$$
for each $(x,s)$. Hence it suffices to show
\begin{equation}\label{tiltedtents}
\sum_is_i\le C\lambda^{-2}\|f\|_2^2\ .
\end{equation}

We first show that if $\I_1\subset \I$ is a collection of disjoint intervals then
$$\sum_{I_1\in \I_1}|I_1|\le C\lambda^{-2}\|f\|_2^2\ ,$$
It is clear that such $\I_1$ has to be countable. Enumerate the intervals in $\I_1$
as $(x_1'-s_1',x_1'+s_1')$, $(x_2'-s_2',x_2'+s_2')$ etc. 
Then we have by choice of the collection $\I$
$$\sum_i s_i'\le \sum_i
\lambda^{-2}
\int_{T_{\alpha,\beta}(x_i',s_i')}|F_\phi(f)(y, t)|^2\, dy 
\frac{dt}t\ .$$
However, the tents $T_{\alpha,\beta}(x_i',s_i')$ are pairwise disjoint,
and hence
$$\sum_i s_i'\le 
\lambda^{-2}
\int_{0}^\infty \int_\R |F_\phi(f)(y, t)|^2\, dy 
\frac{dt}t\ .$$
By Calder\'on's reproducing formula, the latter is bounded
by $$C_\phi \lambda^{-2}\|f\|_2^2\ .$$ 

The above estimate shows in particular that $\sup_{I\in \I} |I|$ is finite, and we may select any $I_1\in \I$ such that $|I_1|>\frac 1 2 \sup_{I\in \I}|I|$. Let $\I'$ be the collection of intervals in $\I$ that does not intersect (or contain) $I_1$. Then select any $I_2 \in \I'$ such that its length is more than half of $\sup_{I\in \I'} |I|$. Iterate this argument we obtain a sequence $I_1,I_2,\dots$ of disjoint intervals in $\I$. We claim that 
$$\bigcup_{I\in \I} I \subset \bigcup_{j} 5I_j$$
here for any $m>0$ we define $mI_j$ to be the interval of length $m|I_j|$ with the same center as $I_j$. Certainly this claim will imply \eqref{tiltedtents}.

Suppose, towards a contradiction, that there exists $I\in \I$ such that $I\not\subset \bigcup_{j} 5I_j$. We first claim that $I$ intersects one of the intervals $I_1,I_2,\dots$. Indeed, since $|I_j|\to 0$ as $j\to\infty$,  there exists $j\ge 1$ such that $|I| > 2|I_{j+1}|$ which means $I$ wasn't available for selection after step $j$, i.e. $I$ has to intersect one of the intervals $I_1,\dots, I_j$. 
Now, let $k \ge 1$ be the smallest index such that $I\cap I_k \ne \emptyset$. It follows that $I$ is available for selection after step $k-1$, and hence $|I_k|\ge \frac 1 2 |I|$ and therefore
$$I\subset 5I_k$$
which contradicts the above assumption.

This proves the lemma.
\endproof


\subsection{Paraproducts and the $T(1)$ theorem}

A classical paraproduct is a bilinear operator, which after 
pairing with a third function becomes a trilinear form that is 
essentially of the type

$$\Lambda(f_1,f_2,f_3)= \int_{\R\times (0,\infty)} 
\prod_{j=1}^3 F_{\phi_j}(f_j)(x,t)\, dx\frac{dt}t$$
with three compactly supported smooth 
functions $\phi_1,\phi_2,\phi_3$ of which two
have vanishing integral while the third does not 
necessarily have vanishing integral.
By symmetry we assume $\phi_1$ and $\phi_2$ to have vanishing integral.
Paraproducts also appear in different forms in the literature,
for example discretized versions of the above integral, or 
versions involving only two embedding maps $F_i$. In the latter case the 
third embedding can typically be inserted after using some manipulations
on the integral expression.

Assuming $f_j$ are bounded, and thus $F_{\phi_j}(f_j)$ are bounded as well,
we obtain by an application of Proposition \ref{measuredomination}
the estimate
$$|\Lambda(f_1,f_2,f_3)| \le C\|\prod_{j=1}^3 F_{\phi_j}(f_j)\|_{L^1(X,\sigma,S_1)}\ .$$
By H\"older's inequality, once the classical one for the sizes and
once Proposition \ref{p.hoelder-energy}, we obtain
$$|\Lambda(f_1,f_2,f_3)| \le C
\|F_{\phi_1}(f_1)\|_{L^{p_1}(X,\sigma,S_2)}\|F_{\phi_2}(f_2)\|_{L^{p_2}(X,\sigma,S_2)}
\|F_{\phi_3}(f_3)\|_{L^{p_3}(X,\sigma,S_\infty)}
$$
for exponents $1< p_1,p_2,p_3\le \infty$. By applying the Carleson
embedding theorems we obtain
$$|\Lambda(f_1,f_2,f_3)| \le C
\|f_1\|_{p_1}\|f_2\|_{p_2}
\|f_3\|_{p_3}\ ,$$
which reproduces classical paraproduct estimates. 
Note that the last estimate does not depend on the $L^\infty$ bounds 
on $f_j$, and thus easily extends to unbounded functions. 
With well known and not too laborous changes in the above arguments 
one can also reproduce classical ${\rm BMO}$ bounds in place of 
$p_1=\infty$ or $p_2=\infty$.

We now state a simplified version of the classical $T(1)$ theorem 
originating in \cite{david-journe}.
\begin{theorem}[$T(1)$ theorem]\label{t1theorem}
Let $\phi$ be some nonzero smooth function supported in
$[-1,1]$ with $\int \phi=0$ and define for $x\in \R$ and $s\in (0,\infty)$
$$\phi_{x,s}(y)=s^{-1}\phi(s^{-1}(y-x))\ .$$
Assume $T$ is a bounded linear operator in $L^2(\R)$ such that for all
$x,y,s,t$
\begin{equation}\label{t1assumption}
|\<T(\phi_{x,s}),\phi_{y,t}\>|\le \frac{\min(t,s)}{\max(t,s,|y-x|)^2}\ .
\end{equation}
Then we have for the operator norm of $T$ the bound
$$\|T\|_{L^2\to L^2}\le C $$ 
for some constant $C$ depending only on $\phi$ and in particular not on $T$.
Moreover, for $1<p<\infty$,
$$\|T f\|_{p}\le C_p \|f\|_p$$ 
for some constant $C_p$ depending only on $\phi$ and $p$.
\end{theorem}
To compare this with more classical formulations of the $T(1)$ theorem,
the assumption (\ref{t1assumption}) is typically deduced from
Calder\'on-Zygmund kernel estimates if $|x-y|>t+s$ and thus the two
test functions $\phi_{x,s}$ and $\phi_{y,t}$ are disjointly supported.
It is deduced from one of the assumptions $T(1)=0$ and $T^*(1)=0$ 
and a weak boundedness assumption if $s$ or $t$ is within a factor 
of $2$ of the maximum of $s$, $t$, and $|y-x|$ and thus the two test 
functions are close. The assumptions $T(1)=0$ and $T^*(1)=0$ can
be obtained from more general assumptions $T(1)\in {\rm BMO}$ and $T^*(1)
\in{\rm BMO}$ by subtracting paraproducts from $T$ first.
A detailed exposition of the $T(1)$ theorem can be found in \cite{stein}.

\begin{proof}

We note from Calder\'on's reproducing formula
$$f=C \int_0^\infty \int_\R  F(x,s) \phi_{x,s} \, dx\, \frac{ds}{s}$$
with a weakly absolutely convergent integral in $L^2$
and $F=F_\phi(f)$ as defined in (\ref{definefphi}). 
Thus we may write with the analoguous notation $G=F_\phi(g)$
$$\<T(f),g\>= $$
$$=\int_0^\infty \int_0^\infty \int_\R \int_\R 
F(x,s) {\<T(\phi_{x,s}),\phi_{y,t}\>} \overline{G(y,t)}
\, dx\,dy \, \frac{ds}
{s}\, \frac{dt}{t}\ .$$

Here we implicitly used boundedness of $T$ and the Schwarz kernel theorem to move $T$ inside the integral representation of $f$.
Note that we have again expressed the form $\<T(f),g\>$ in terms of the functions $F$ and $G$ on the outer space $X$,
which leads towards the use of embedding theorems. However, we cannot apply H\"older's inequality directly,
but we first have to suitably express the double integral over the space $X$ as superposition of single integrals over $X$.

Set $$r:=\max(s,t,|y-x|)\ .$$
We split the domain of integration into the two regions
$r> |x-y|$  and $r=|x-y|$ and estimate the two integrals 
separately. 
Splitting the first region further into two symmetric 
regions (overlapping in a set of measure zero), we may 
restrict attention to the region $s=r$.
We estimate the integral over this region by
$$|\int_0^\infty \int_\R 
\int_{0}^s \int_{x-s}^{x+s} 
F(x,s) 
{\<T(\phi_{x,s}),\phi_{y,t}\>}
\overline{G(y,t)} \, dy\,\frac{dt}t
               \, dx\, \frac{ds}s|$$
$$\le C
\int_0^\infty \int_\R 
\int_{0}^s \int_{x-s}^{x+s} 
|F(x,s) 
\overline{G(y,t)}| \, dy\,{dt}
               \, dx\, \frac{ds}{s^3}$$
$$= C
\int_0^1\int_{-1}^1
\int_0^\infty \int_\R 
|F(x,s) 
\overline{G(x+\alpha s,\beta s)} | \,
               \, dx\, \frac{ds}{s}\,
\, d\alpha \, {d\beta }\ .
$$
In the last line we have changed variables setting 
$y-x=\alpha s$ and $t=\beta s$.
Setting ${G}_{\alpha,\beta}(x,s)=G(x+\alpha s,\beta s)$
we estimate the last display, using Propositions \ref{measuredomination}
and outer H\"older's inequality with dual exponents $1<p,p'<\infty$,
Proposition \ref{p.hoelder-energy}, 
$$\le C
\int_0^1\int_{-1}^1
\|F G_{\alpha,\beta}\|_{L^1(X,\sigma,S_1)}
\, d\alpha \, {d\beta }$$
$$
\le C
\int_0^1\int_{-1}^1
\|F\|_{L^p(X,\sigma,S_2)}
\|G_{\alpha,\beta}\|_{L^{p'}(X,\sigma,S_2)}
\, d\alpha \, {d\beta }\ .
$$
The norm of $F$ can be estimated by
Theorem \ref{carlesonembedding}, while the norm of $G$
can be estimated by Lemma \ref{tiltedembedding}.
Hence we can estimate the last display by
$$\le C
\int_0^1\int_{-1}^1
\beta^{-1/p'}\|f\|_{p}\|g\|_{p'}
\, d\alpha \, {d\beta }
\le 
 C\|f\|_{p}\|g\|_{p'}\ .
$$

The region $r=|y-x|$ we also split into symmetric regions, \sout{first}
restricting to $r=y-x$  and $r=x-y$.  By symmetry, it suffices to estimate the region $r=y-x$. We now split further into $t\le s$ and $s\le t$.

We obtain for the subregion $t\le s$ the estimate
$$|\int_{0}^\infty\int_\R \int_{0}^r \int_0^s  F(x,s) 
 \< T(\phi_{x,s}),\phi_{x+r,t}\>
\overline{G(x+r,t)} \frac{dt}t \frac {ds}s \, dx\, {dr}| $$
$$\le 
\int_0^\infty \int_\R \int_{0}^r \int_0^s  |F(x,s) 
\overline{G(x+r,t)}| {dt} \frac{ds}s \, dx\, \frac{dr}{r^2} $$
$$=
\int_{0}^1 \int_0^\alpha  
\int_0^\infty \int_\R 
|F(x,\alpha r) 
\overline{G(x+r,\beta r)} |
\, dx\, \frac{dr}r 
{d\beta } \frac{d\alpha }\alpha  $$
$$=
\int_{0}^1 \int_0^\alpha  
\int_0^\infty \int_\R 
|F_{0,\alpha}(x,r) 
\overline{G_{1,\beta}(x,r)} |
\, dx\, \frac{dr}r 
{d\beta } \frac{d\alpha }\alpha \ , $$
where we have used the notation $F_{0,\alpha}$ and $G_{1,\beta}$ as above.
We use Lemma \ref{tiltedembedding} twice to estimate the last display
by
$$\le C
\int_{0}^1 \int_0^\alpha  
\alpha^{\epsilon-1/p}\beta^{\epsilon-1/p'} \|f\|_p\|g\|_{p'}
{d\beta } \frac{d\alpha }\alpha  \le  C\|f\|_p\|g\|_{p'}   \ .
$$
The subregion $s\le t$ could be estimated similarly. This concludes the proof of the  $L^2$ and $L^p$
estimate of  Theorem \ref{t1theorem}.
\end{proof}

We conclude this section by pointing at an alternative approach to 
Calder\'on Zygmund operators used in A. Lerner's work \cite{lerner}, who
essentially controls a Calder\'on Zygmund operators by a superposition of ``sparse'' operators.
These sparse operators lend themselves to an application of an outer H\"older inequality 
with spaces $ L^{\infty}(X,\sigma,S_1)\times L^{p}(X,\sigma,S_\infty)\times L^{p'}(X,\sigma,S_\infty)$
in lieu of the above $ L^{p}(X,\sigma,S_2)\times L^{p'}(X,\sigma,S_2)$
or implicit $ L^{\infty}(X,\sigma,S_\infty)\times L^{p}(X,\sigma,S_2)\times L^{p'}(X,\sigma,S_2)$.

\section{Generalized Tents and Carleson Embedding}
\label{gentents}

In this section we introduce a new outer measure space whose
underlying set is the upper three space. The extra dimension
relative to the classical tent spaces is a frequency parameter, which 
arises due to modulation symmetries in problems of time-frequency analysis.
In contrast, the upper half plane merely represents dilation and translation 
symmetries. The generalized Carleson embedding theorem below is new, though 
its proof is an adaption of standard recipes in time-frequency analysis. 
The novelty lies in the concise formulation of an essential part of 
time-frequency analysis, and in the absence of any discretization
in the formulation of Theorem 
\ref{gen.carl.emb}.

This section is the most technical one of the present paper, as it is devoted 
to a proof of Theorem \ref{gen.carl.emb} and its discrete variant, Theorem  
\ref{gen.carl.emb.disc}. We point out that the application 
of Theorem \ref{gen.carl.emb} to the bilinear Hilbert transform
discussed in the final section can be understood without detailed reading of 
the proof in the present section.

Let $X$ be the space $\R\times \R\times (0,\infty)$ with
the usual metric as subspace of $\R^3$. 
Let $0<|\alpha|\le 1$ and $|\beta|\le 0.9$ be two real parameters
and define for a point $(x,\xi,s)$ in $X$ the generalized tent
\begin{equation}\label{gentent}
T_{\alpha,\beta}(x,\xi,s):=\{(y,\eta,t)\in X: t\le s, |y-x|\le s-t, |\alpha(\eta-\xi)+\beta t^{-1}|\le t^{-1}\}\ .
\end{equation}
For a first understanding the reader may focus on the
example $\alpha=1$ and $\beta=0$.  In this case the condition on the frequency variable $\eta$ becomes $-t^{-1} \le \eta-\xi \le t^{-1}$ which is symmetric around $\xi$, as can be seen below. The general case with other $(\alpha,\beta)$ leads to a condition $At^{-1}\le \eta - \xi \le Bt^{-1}$ for some $A<0<B$ depending on $\alpha,\beta$, and will correspond to an asymmetric variant of the Figure below.

\setlength{\unitlength}{0.4mm}
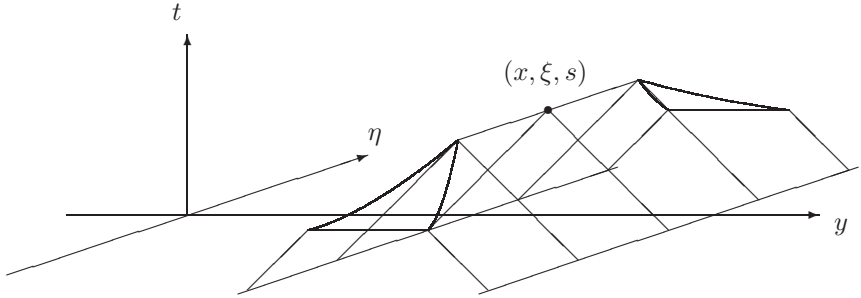
\begin{figure}[ht]
\caption{The generalized tent $T_{1,0}(x,\xi,s)$}
\begin{picture}(300,120)
\put(20,40){\vector(1,0){250}}
\put(0,20){\vector(3,1){120}}
\put(60,40){\vector(0,1){60}}
\put(275,35){$y$}
\put(55,105){$t$}
\put(120,65){$\eta$}
\put(77,14){\line(3,1){126}}
\put(157,14){\line(3,1){126}}
\put(150,65){\line(3,1){60}}
\put(140,35){\line(1,1){40}}
\put(180,75){\line(1,-1){40}}
\put(170,45){\line(1,1){40}}
\put(210,85){\line(1,-1){40}}
\put(110,25){\line(1,1){40}}
\put(150,65){\line(1,-1){40}}
\put(80,15){\line(1,1){20}}
\put(160,15){\line(-1,1){20}}
\put(100,35){\line(1,0){40}}
\put(200,55){\line(1,1){20}}
\put(280,55){\line(-1,1){20}}
\put(220,75){\line(1,0){40}}
\put(180,75){\circle*{3}}
\put(165,85){$(x,\xi, s)$}
\qbezier(100,35)(120,40)(150,65)
\qbezier(210,85)(215,78)(220,75)
\qbezier(140,35)(145,40)(150,65)
\qbezier(210,85)(235,78)(260,75)
\end{picture}
\end{figure}

The projection of the generalized tent onto the first two variables is a 
classical tent as in Example 3. We are only concerned with generalized
tents in this section and will omit the adjective ``generalized'' when
referring to $T_{\alpha,\beta}(x,\xi,s)$.
The collection $\E$ of all tents generates an outer measure if we set
$$\sigma(T_{\alpha,\beta}(x,\xi,s))=s\ .$$ By a similar argument as in Example 3, 
$\sigma$ satisfies \eqref{ecoverc}, and hence the outer 
measure $\mu$ is an extension of the function $\sigma$ on $\E$.

To define a size on Borel functions on $X$, we use further auxiliary 
tents
\begin{equation}\label{definegent}
T^b(x,\xi,s):=\{(y,\eta,t)\in X: t\le s, |y-x|\le s-t, |\eta-\xi|\le b t^{-1}\}\ .
\end{equation}
For $0<b< 1$ and a Borel measurable function $F$  on $X$ we define 
\begin{equation}\label{definegens}
S^{b}(F)(T_{\alpha,\beta}(x,\xi,s)):=
\end{equation}
$$
(s^{-1}\int_{T_{\alpha,\beta}(x,\xi,s)\setminus T^b(x,\xi,s)}|F(y,\eta,t)|^2 \, dy\, 
d\eta  \,{dt})^{1/2}
+ \sup_{(y,\eta,t)\in T_{\alpha,\beta}(x,\xi,s)}|F(y,\eta,t)|\ .$$
One easily checks that this size satisfies the properties required
in Definition \ref{d.size}. 
The size $S^b$ increases as $b$ decreases.

The following is a version of a Carleson embedding theorem
in the setting of generalized tents.
We normalize the Fourier transform
of a Schwartz function $\phi$ on the real line as
$$\widehat{\phi}(\xi)=\int_\R e^{-i\xi x}\phi(x)\, dx\ .$$

\begin{theorem}[Generalized Carleson embedding]\label{gen.carl.emb}

Let $0<|\alpha|\le 1$ and $|\beta|\le 0.9$.
Let $0<b\le 2^{-8}$. Let $\phi$ be a Schwartz function 
with Fourier transform $\widehat{\phi}$ supported in $(-2^{-8}b,2^{-8}b)$, 
and let $2\le p\le \infty$.  
Define for $f\in L^p(\R)$ the function $F$ on $X$ by
$$F(y,\eta,t):=\int_\R f(x) e^{i\eta (y-x)} t^{-1}\phi(t^{-1}(y-x)) \, dx\ .$$
There is some constant $C$ depending only on $\alpha$, $\beta$, $b$, $\phi$, and $p$,
such that if $p > 2$, 
$$\|F\|_{\L^{p}(X,\sigma,S^b)}\le C \|f\|_p\ ,$$
and if $p=2$, 
$$\|F\|_{\L^{2,\infty}(X,\sigma,S^b)}\le C \|f\|_2\ .$$
\end{theorem}

By symmetry it is no restriction to assume $0<\alpha$ and we shall do so.


The dependence of the constant $C$ on the function $\phi$, 
conditioned on the fixed support condition on $\widehat{\phi}$, 
factors as dependence on the constant
$$\sup_{x}\left[ |\phi(x)| (1+|x|)^{3} +  |\phi'(x)| (1+|x|)^{2}\right]\ .$$
We do not claim that this explicit regularity of $\phi$ is sharp
for the above theorem to hold. 

From now on we fix the parameters $\alpha$ and $\beta$, and for simplicity
of notation write $T$ for $T_{\alpha,\beta}$.

It is convenient to work with a discrete variant of
Theorem \ref{gen.carl.emb}. Fix the parameter 
$0<b\le 2^{-8}$. We introduce the discrete subset $X_\Delta$ 
of points $(x,\xi,s)\in X$  such that there exist integers 
$k,n,l\in \Z$ with $$x=2^{k-4} n,\  \xi=2^{-k-8} b l,\  s=2^k\ .$$ 
We denote by $\E_\Delta$ the collection of all tents $T(x,\xi,s)$
with $(x,\xi,s)\in X_\Delta$. This is a discrete subcollection of
$\E$. However, each tent in $\E_\Delta$ by itself still forms a 
continuum in $X$.

We generate an outer measure $\mu_\Delta$ using $\E_\Delta$ 
as generating collection, setting as before $\sigma_\Delta(T(x,\xi,s))=s$ 
for each tent in $\E_\Delta$.

The following lemma will be used to relate this new measure to the 
previous one.
\begin{lemma}\label{discretizationlemma}
If $(x',\xi',s')\in X$, then there exists a $(x,\xi,s)\in X_\Delta$ 
such that the tent $T(x,\xi,s)$ contains $(x',\xi',s')$ ``centrally'' 
in the sense
$$2^{-3}s<s'\le 2^{-2}s \ ,$$
$$|x'-x|\le 2^{-4} s\ ,$$
$$|\xi'-\xi|\le 2^{-8}b s^{-1}\ .$$
Moreover, there exist two points $(x,\xi_-,s)\in X_\Delta$ and 
$(x,\xi_+,s)\in X_\Delta$ so that the corresponding tents contain 
$(x',\xi',s')$ centrally and satisfy
$$T(x',\xi',s')\subset T(x,\xi_-,s)\cup T(x,\xi_+,s)\ ,$$
$$
T(x',\xi',s')\cap T^b(x,\xi_-,s)\cap T^b(x,\xi_+,s) \subset T^b(x',\xi',s')
\ .$$
\end{lemma}
\begin{proof}
The interval $[2^2s',2^3s')$ contains a unique point
of the form $2^k$ with $k\in \Z$. We set $s=2^k$.
Then there is a point $x$ of the form $2^{k-4}n$
with some $n\in \Z$ such that $|x-x'|\le 2^{k-4}=2^{-4}s$. 
Likewise, there is a point $\xi$ of the form $2^{-k-8}bl$ with
$l\in \Z$ such that $|\xi-\xi'|\le 2^{-k-8}b=2^{-8}bs^{-1}$.

Informally, this point $\xi$ may be chosen on either
side of $\xi'$. Precisely, we may choose
$\xi_-\le \xi'$ and $\xi_+\ge \xi'$ with
$|\xi_--\xi'|\le 2^{-8}bs^{-1}$ and $|\xi_+-\xi'|\le 2^{-8}bs^{-1}$.
If $(y,\eta, t)\in  T(x',\xi',s')$, then we have
$$t\le s'\le 2^{-2} s\ ,$$
$$|y-x|\le |y-x'|+|x'-x|\le s'-t+2^{-4}s\le s-t\ .$$
If in addition $\alpha(\eta- \xi')+\beta t^ {-1}\ge 0$, then  (recall that $\alpha>0$) 
$$-t^{-1}\le \alpha(\xi'-\xi_+)\le \alpha(\eta-\xi_+)+\beta t^{-1} \le \alpha(\eta-\xi')+\beta t^{-1}\le t^{-1}\ ,$$
while if in addition $\alpha(\eta- \xi')+\beta t^ {-1}\le 0$, then
$$- t^{-1}\le \alpha(\eta-\xi')+\beta t^{-1}\le 
\alpha(\eta-\xi_-)+\beta t^{-1}\le \alpha(\xi'-\xi_-) \le t^{-1}\ .$$
Hence $(y,\eta,t)\in T(x,\xi_-,s)\cup T(x,\xi_+,s)$.
Now let in addition $(y,\eta, t)$ be an element of 
$T^b(x,\xi_-,s)\cap T^b(x,\xi_+,s)$.
If $\eta\ge \xi'$, then
$$ -bt^{-1} < 0  \le \eta-\xi' \le \eta-\xi_- \le bt^{-1}\ ,$$
while if $\eta\le \xi'$, then 
$$-bt^{-1}\le  \eta-\xi_+\le \eta-\xi' \le  0 <  bt^{-1} \ .$$
Hence $(y,\eta,t)\in T^b(x',\xi',s')$.
This completes the proof of the lemma.
\end{proof}

As a consequence of this lemma, if $T$ is a tent in $\E$, then we find
two tents $T^+$, $T^-$ in $\E_\Delta$ such that
$$T\subset T^+\cup T^-\ ,$$
$$\sigma_\Delta(T^+)+\sigma_\Delta(T^-)\le C\sigma(T)\ .$$
This implies for every subset $X'\subset X$
$$\mu(X')\le \mu_\Delta(X')\le C \mu(X')\ .$$
Hence the outer measures $\mu$ and $ \mu_\Delta$ are equivalent.

Moreover, we have
for the same tents and every Borel function $F$
$$S^b(F)(T)\le C [S_\Delta^b (F)(T^+)+S_\Delta^b(F)(T^-)]\ ,$$
where we have defined 
$$S^b_\Delta(F)(T'):= S^b (F)(T')\ .$$
for any tent $T'$ in $\E_\Delta$. 

This implies for every $1\le p\le \infty$
$$
C^{-1}\L^{p}(X,\sigma,S^b)
\le  \L^{p}(X,\sigma_\Delta,S^b_\Delta)\le
\L^{p}(X,\sigma,S^b)\ ,
$$
$$
C^{-1}\L^{p,\infty}(X,\sigma,S^b)
\le  \L^{p,\infty}(X,\sigma_\Delta,S^b_\Delta)\le
\L^{p,\infty}(X,\sigma,S^b)\ .
$$

Hence Theorem \ref{gen.carl.emb} is equivalent to the following
discrete version.

\begin{theorem}[Generalized Carleson embedding, discrete version]
\label{gen.carl.emb.disc}

Let $0<\alpha\le 1$ and $-0.9\le \beta\le 0.9$.
Let $0<b\le 2^{-8}$. Let $\phi$ be a Schwartz function 
with Fourier transform $\widehat{\phi}$ supported in the interval $(-2^{-8}b,2^{-8}b)$, 
and let $2\le p\le \infty$.  
Define for $f\in L^p(\R)$ the function $F$ on $X$ by
$$F(y,\eta,t):=\int_\R f(x) e^{i\eta (y-x)} t^{-1}\phi(t^{-1}(y-x)) \, dx\ .$$
There is some constant $C$ depending only on $\alpha$, $\beta$, $b$, $\phi$, and $p$,
such that if $p \neq 2$, 
$$\|F\|_{\L^{p}(X,\sigma_\Delta,S^b_\Delta)}\le C \|f\|_p\ ,$$
and if $p=2$, 
$$\|F\|_{\L^{2,\infty}(X,\sigma_\Delta,S^b_\Delta)}\le C \|f\|_2\ .$$
\end{theorem}

\begin{proof}[Proof of Theorems \ref{gen.carl.emb}
and \ref{gen.carl.emb.disc}.] 
Since both theorems are equivalent, we will only prove 
the discrete version, Theorem \ref{gen.carl.emb.disc}.
Hence we will only work with the discrete quantities
$\mu_\Delta$ and $S^b_\Delta$ and for simplicity of 
notation omit the subscribt $\Delta$.
Since $b$ is fixed, we also denote $S:= S^b$.

The theorem follows  by Marcinkiewicz interpolation, Proposition
\ref{p.marcinkiewicz}, between the end point
cases $p=2$ and $p=\infty$.

\subsection{The endpoint $p=\infty$}

We need to prove that for every $(x,\xi,s)\in X_\Delta$ and
every $f\in L^\infty(\R)$ we have
$$S(F)(T(x,\xi,s))\le C\|f\|_\infty\ \ .$$
The size $S$ is defined as a sum of an $L^2$ portion and an 
$L^\infty$ portion. It suffices to estimate both portions separately.
Note that for all $y,\eta,t$ we trivially
have $|F(y,\eta,t)|\le \|f\|_\infty \|\phi\|_1$ and this 
establishes the desired bound on the $L^\infty$ portion of $S$.

To estimate the $L^2$ portion of the size we first establish
the estimate
\begin{equation}\label{l2tentbound}
\int_{T(x,\xi,s)\setminus T^b(x,\xi,s)}|F(y,\eta,t)|^2 \, dy\,d\eta\, dt\le 
C\|{f}\|_2^2
\end{equation}
for every function $f\in L^2(\R)$.
Fix such a function $f$, we may assume by normalization that $\|f\|_2=1$.
Replacing the domain of integration by a larger region we can 
estimate the left-hand-side of (\ref{l2tentbound}) by
$$  
\int_0^\infty \int_\R \int_{bt^{-1}\le |\eta-\xi|\le 2\alpha^{-1} t^{-1}} |F(y,\eta,t)|^2\, d\eta \, dy \, {dt}\ .$$
It suffices to estimate the integral over the region where 
$\eta> \xi$ , since by symmetry there is an analoguous
estimate for the integral over region $\eta<\xi$.
We replace the integration variable $\eta$ by $\gamma$ 
such that $\eta-\xi=\gamma t^{-1}$. Using Fubini we are reduced to estimating
$$  
\int_{b}^{2\alpha^{-1}}
\int_0^\infty \int_\R  |F(y,\xi+\gamma t^{-1},t)|^2 \, dy \, \frac {dt}t \, d\gamma\ .$$
We first estimate the inner double integral
for fixed $\gamma$.

Define for each $y,\gamma,t$ the bump function $\phi_{y,\gamma,t}$ by 
$$\phi_{y,\gamma,t}(x)= e^{-i(\xi+\gamma t^{-1}) (y-x)} t^{-1}\overline{\phi(t^{-1}(y-x))} \ . $$
We are interested in the region $\gamma\ge b$, where the
modulated function $\phi_{y,\gamma,t}e^{-i\xi.}$
has integral zero by support consideration of $\widehat{\phi}$. 
Hence its primitive is absolutely 
integrable with good bounds, which we will use later when applying
partial integration.

We have 
$$( \int _0^\infty \int_\R 
|\<f,\phi_{y,\gamma ,t}\> |^2 \, dy\, \frac{dt}t)^2
$$
$$\le \|   \int_0^\infty \int_\R        \<f,\phi_{y,\gamma ,t}\> 
\phi_{y,\gamma ,t} \, dy\, \frac{dt}t\|_2^2 $$
$$
\le    \int _0^\infty \int_\R     \int _0^\infty \int_\R     
|\<f,\phi_{y,\gamma ,t}\> 
\<\phi_{y,\gamma ,t}, \phi_{z,\gamma ,r} \>
\<\phi_{z,\gamma ,r},f\>| 
\, dz\, \frac{dr}r
\, dy\, \frac{dt}t \ .
$$
Estimating the smaller  of the inner products with $f$ by the  larger
one  and using symmetry we may estimate this by
\begin{equation}\label{smallbigf}
\le    2\int _0^\infty \int_\R |\<f,\phi_{y,\gamma ,t}\>|^2 
[    \int _0^\infty \int_\R     
|\<\phi_{y,\gamma ,t}, \phi_{z,\gamma ,r} \>| 
\, dz\, \frac{dr}r 
]\, dy\, \frac{dt}t \ .
\end{equation}
 

We consider the inner double integral of (\ref{smallbigf}).
Considering first the region $t\le r$ and doing partial integration in
the inner product 
$$
\<\phi_{y,\gamma ,t}, \phi_{z,\gamma ,r} \>=
\<\phi_{y,\gamma ,t}e^{-i\xi .}, \phi_{z,\gamma ,r} e^{-i\xi.}\>\ ,$$
integrating the first and differentiation the second bump function, we estimate
 the integral over this region by 
$$C\int_\R  \int_t^\infty 
\int_\R 
 (1+|t^{-1}(y-x)|)^{-2}
 r^{-2}(1+|r^{-1}(z-x)|)^{-2} dx  \frac{dr}r\, dz$$
$$\le C\int_\R  \int_t^\infty 
 (1+|t^{-1}(y-x)|)^{-2}
 r^{-2} {dr}\, dx$$
$$\le C\int_\R   
 t^{-1}(1+|t^{-1}(y-x)|)^{-2}
 \, dx\le C\ .$$
In the region $t\ge r$ we do partial integration in reverse,
differentiating the first and integrating the second bump function, 
to obtain the estimate for the integral over this region by
$$C\int_\R  \int_0^t \int_\R 
 t^{-2}(1+|t^{-1}(y-x)|)^{-2}
 (1+|r^{-1}(z-x)|)^{-2} dx \frac{dr}r\, dz$$
$$\le C\int_\R  \int_0^t 
 t^{-2}(1+|t^{-1}(y-x)|)^{-2}
  {dr}\, dx$$
$$\le C\int_\R  
 t^{-1}(1+|t^{-1}(y-x)|)^{-2}
  \, dx\le C\ .$$
Inserting these two estimates into (\ref{smallbigf}) gives
$$( \int _0^\infty \int_\R 
|\<f,\phi_{y,\gamma ,t}\> |^2 \, dy\, \frac{dt}t)^2
\le C
\int _0^\infty \int_\R 
|\<f,\phi_{y,\gamma ,t}\> |^2 \, dy\, \frac{dt}t\ ,$$
which proves (\ref{l2tentbound}).

We note that if we restrict the integral on the left hand side 
of (\ref{l2tentbound}) to the region $\eta>\xi$, 
we may improve the bound on the right-hand-side to 
\begin{equation}\label{fxibound}
C\|\widehat{f}1_{(\xi,\infty)}\|_2^2
\ .
\end{equation}
This follows simply by support considerations on the Fourier
transform side.

Now assume that $f\in L^\infty(\R)$ and write
$f=f_1+f_2$ where 
$$f_1=f1_{[x-2s,x+2s]}\ .$$
By linearity we may split $F=F_1+F_2$ correspondingly.
We have $\|f_1\|_2^2\le Cs\|f\|_\infty^2$, so by the above
$L^2$ bound we have
$$ (s^{-1}\int_{T(x,\xi,s)\setminus T^b(x,\xi,s)}|F_1(y,\eta,t)|^2 \, dy\, d\eta  \,{dt})^{1/2}
\le C\|f\|_\infty\ .$$
It remains to prove the analoguous estimate for $F_2$.
But for $y\in [x-s,x+s]$ and $t<s$ we have
$$F_2(y,\eta, t)\le \int_{[-s, s]^c} |f_2(y-z)| t^{-1} |\phi(t^{-1}z)|\, dz \le C (t/s)\|f\|_\infty\, $$
where we have crudely estimated the integral of the tail of $\phi$. But then 
$$ (s^{-1}\int_{T(x,\xi,s)\setminus T^b(x,\xi,s)}|F_2(y,\eta,t)|^2 \, dy\, d\eta  \,{dt})^{1/2}$$
$$\le  C\|f\|_\infty (s^{-1}\int_{0}^s \int_{\xi-2\alpha^{-1}t^{-1}}^{\xi+2\alpha^{-1}t^{-1}} 
\int_{x-s}^{x+s} (t/s)^{2} \, dy\, d\eta  \,{dt})^{1/2}
\le C\|f\|_\infty\ .$$
This completes the proof of the endpoint $p=\infty$ of 
Theorem \ref{gen.carl.emb}.

\subsection{The endpoint $p=2$}

We need to find for each $\lambda>0$ 
a collection $\Q \subset X_\Delta$
such that 
$$\sum_{(x,\xi,s)\in \Q} s\le C\lambda^{-2} \|f\|_2^2$$
and
for every $T'\in \E_\Delta$
we have 
\begin{equation}\label{small2inftysize}
S(F 1_{X\setminus E})(T')\le \lambda\ ,
\end{equation}
where $E=\bigcup_{(x,\xi,s)\in \Q} T(x,\xi,s)$.

We first reduce to the special case that the support of $\widehat{f}$ 
is compact.  Choose an unbounded  monotone increasing sequence $\xi_k$, 
$k=0,1,2,\dots$ with $\xi_0=0$ such that for $f_k$ defined by
$\widehat{f}_k=\widehat{f}[1_{(-\xi_{k},-\xi_{k-1})}+1_{(\xi_{k-1},\xi_{k})}]$ 
we have
$$\|f_k\|_2\le C2^{-10k}\|f\|_2\ .$$
Applying the special case to each of the functions $f_k$
with $\lambda_k=2^{-k}\lambda $ we obtain corresponding 
collections $\Q_k$.
Then clearly
$$\sum_{k=1}^\infty  \sum_{(x,\xi,s)\in \Q_k} s\le 
C \sum_k (2^{-k}\lambda)^{-2} 2^{-20k} \|f\|_2^2
\le C \lambda^{-2} \|f\|_2^2\ .$$
If $E$ denotes the union of all $T(x,\xi,s)$
with $(x,\xi,s)\in \bigcup_k \Q_k$, then 
by countable subadditivity of the size $S$ we have for every $T'\in \E_\Delta$
$$ S(F1_{X\setminus E})(T')\le \sum_{k=1}^\infty S(F_k1_{X\setminus E})(T')
\le \sum_{k=1}^\infty  2^{-k}\lambda \le \lambda\ .$$
This completes the reduction to the case that $\widehat{f}$ 
has compact support, and we shall henceforth assume 
compact support of $\widehat{f}$. 

By scaling of outer Lebesgue spaces we may
assume $\|f\|_2=1$. Fix $\lambda>0$. 
We first set out to cover all points $(y,\eta,t)\in X$ 
with $|F(y,\eta,t)|> \lambda$ with tents.
Note that there is an a priori upper bound on $t$ for any
such point since by Cauchy-Schwarz we have directly from the definition of $F$:
$$|F(y,\eta,t)|\le C t^{-1/2} \|f\|_2 \|\phi\|_2\le Ct^{-1/2}\ .$$
Assume there is a point $(y,\eta,t)$
with $|F(y,\eta,t)|> \lambda$, then by Lemma
\ref{discretizationlemma} we find a tent $T(x,\xi,s)$ centrally 
containing the point $(y,\eta,t)$.
Because of the upper bound on $t$ and since $(x,\xi,s)\in X_\Delta$ and 
therefore $s=2^k$ for some integer $k$, we may choose $(y,\eta,t)$ 
and $(x,\xi,s)$ such that $s$ is maximal. Denote these points by
$(y_1,\eta_1,t_1)$ and $(x_1,\xi_1,s_1)$ and the tent 
$T(x_1,\xi_1,s_1)$ by $T_1$.

We continue to select tents by iterating this procedure.
Assume that we have already chosen points $(y_k,\eta_k,t_k)\in X$ and
tents $T_k=T(x_k,\xi_k,s_k)$ for all $1\le k< n$.
Assume there is a point $(y,\eta,t)$ 
with $|F(y,\eta,t)|> \lambda$ 
not contained in the union of the tents
$T_k$ with $1\le k<n$. 
Then we choose such a point $(y_n,\eta_n,t_n)$ 
and a tent $T_n=T(x_n,\xi_n,s_n)$ centrally containing
$(y_n,\eta_n,t_n)$ such that $s_n$ is maximal.
We have $|F(y_n,\eta_n,t_n)|> \lambda$ and
$$(y_n,\eta_n,t_n)\not\in \bigcup_{k=1}^{n-1} T_k\ .$$

We claim that
\begin{equation}\label{count.tiles}
\sum_{k=1}^n s_k \le C\lambda^{-2}\ .
\end{equation}
To see the claim, let $K_m$ be the set of indices $k$ with $1\le k\le n$ such that
$$
2^m\lambda 
\le  |F(y_k,\eta_k,t_k)|\le  2^{m+1} \lambda \ .$$ 
Then we have
$$\sum_{k=1}^n s_k
\le C \sum_{m=0}^\infty 
2^{-2m} \lambda^{-2}
\sum_{k\in  K_m} t_k |F(y_k,\eta_k,t_k)|^2\ .$$
The claim (\ref{count.tiles}) will follow if we show for fixed $m\ge 0$
\begin{equation}\label{seta}
\sum_{k\in  K_m} t_k |F(y_k,\eta_k,t_k)|^2 \le C \ .
\end{equation}
Define 
\begin{equation}\label{wavepacket}
\phi_k(x):=\phi_{y_k,\eta_k,t_k}(x):=e^{-i\eta_k (y_k-x)}t_k^{-1/2}\overline{\phi(t_k^{-1}(y_k-x))}\ ,
\end{equation}
so that
$$t_k^{1/2}F(y_k,\eta_k,t_k)=\<f,\phi_k\>\ .$$
Let $A$ denote the left-hand-side of (\ref{seta}).
Assume we can show for every $k\in K_m$
\begin{equation}\label{modifiedschurc}
\sum_{l\in K_m: s_l\le s_k} (t_l/t_k)^{1/2}|\<\phi_k,\phi_l\>|\le C \ .
\end{equation}
Then we obtain
$$A^2 \le \|\sum_{k\in K_m} \<f,\phi_k\>\phi_k\|_2^2 $$
$$\le \sum_{k,l\in K_m} \<f,\phi_k\>\<\phi_k,\phi_l\> \<\phi_l,f\>$$
$$\le 2\sum_{k,l\in K_m: s_l\le s_k} |\<f,\phi_k\>\<\phi_k,\phi_l\> \<\phi_l,f\>| $$

$$\le C\sum_{k,l\in K_m: s_l\le s_k} (t_l/t_k)^{1/2}
|\<f,\phi_k\>|^2 |\<\phi_k,\phi_l\>| \le C A \ .$$
Here in the passage from the penultimate to ultimate line
we have used that $k,l\in K_m$ and hence  $t_k^{-1/2}|\<f,\phi_k\>|$ and
$t_l^{-1/2}|\<f,\phi_l\>|$ are within a factor of $2$ of each other
and in the last line we have used (\ref{modifiedschurc}). 
Dividing by $A$ on both sides of the displayed inequality 
we have reduced the proof of the desired estimate (\ref{seta})
to the proof of (\ref{modifiedschurc}).

To prove (\ref{modifiedschurc}), fix $k$. 
If $l\in K_m$ with $s_l\le s_k$ such that
$\<\phi_k,\phi_l\>\neq 0$ then
the supports of $\widehat{\phi}_k$ and
$\widehat{\phi}_l$ overlap and hence there
are numbers
 $- 2^{-8} \le  \gamma,\delta \le 2^{-8}$
such that
$$
\eta_l+\delta b t_l^{-1}=\eta_k+\gamma b t_k^{-1}
\ .$$
Now suppose that there is another such $l'$ and we have 
analoguously
$$
\eta_{l'}+\delta' b t_{l'}^{-1}=\eta_k + \gamma'  b t_k^{-1}
\ .$$
Assume without loss of generality that $T_{l}$ is selected prior to $T_{l'}$ 
and thus $s_k^{-1}\le s_{l}^{-1}\le  s_{l'}^{-1}$. 
Then we have by central containment of $(y_l,\eta_l,t_l)$ in $T_l$ 
$$|\alpha (
\eta_{l'}- \xi_l)+\beta t_{l'}^{-1}|\le 
|\alpha (
\eta_{l}-\xi_l)|+|\alpha(\eta_l-\eta_{l'})|+|\beta t_{l'}^{-1}|
$$
$$\le b t_{l}^{-1}+|\eta_l-\eta_{l'}| + |\beta t_{l'}^{-1}|
\ .$$
Now using the information from the support of the bump functions
we may estimate the latter by
$$\le b t_l^{-1} 
+\delta  b t_l^{-1}+\delta' b t_{l'}^{-1}
+\gamma b t_k^{-1}+\gamma' b t_k^{-1}+|\beta t_{l'}^{-1}|
\le t_{l'}^{-1}
\ .$$
Here we have used again the central containment to estimate
the inverse powers of $t_k$ and $t_l$ by that of $t_{l'}$.

This implies that $[x_l-2^{-8}s_l,x_l+2^{-8}s_l]$ and 
$[x_{l'}-2^{-8}s_{l'},x_{l'}+2^{-8}s_{l'}]$ 
are disjoint. For if they were not disjoint, then, since $s_{l'}\le s_l$,
we would conclude
$$|y_{l'}-x_l|\le |y_{l'}-x_{l'}|+|x_{l'}-x_{l}|$$
$$\le 2^{-4}s_{l'}+2^{-4}s_l < s_l-t_{l'}$$
and together with the previous estimate for $\xi_l-\eta_{l'}$  this implied that
the point $(y_{l'},\eta_{l'},t_{l'})$ was in the tent $T_l$,  
contradicting the choice of this point.

 The argument above in particular shows that $|x_k-x_l|\ge 2^{-8}s_k$.   
Let $\overline{x}$ be the midpoint of $x_l$ and $x_k$ and
let $H_l$ and $H_{k}$ be the half lines emanating from 
the midpoint containing $x_l$ and $x_{k}$ respectively.
We then have 
$$|\<\phi_k,\phi_l\>|\le \|\phi_k\|_{L^1(H_k)}\|\phi_l\|_{L^\infty(H_k)}+
\|\phi_k\|_{L^\infty(H_l)}\|\phi_l\|_{L^1(H_l)}\ .$$
Thanks to the rapid decay of the wave packets and
$s_l\le s_k$ 
  and the fact  that $|x_k-x_l| \ge 2^{-8}s_k$ 
we can estimate the last display by
$$C (s_ls_k)^{-1/2}\int (1+(\frac{x-x_k}{s_k})^2
)^{-2} 1_{[x_l-2^{-8}s_l,x_l+2^{-8}s_l]}(x)\, dx\ .$$
 
By disjointness of the intervals $[x_l-2^{-8}s_l,x_l+2^{-8}s_l]$ 
for different $l$ we obtain
$$\sum_{l\in K_m,s_l\le s_k}
 s_l^{1/2} |\<\phi_k,\phi_l\>|$$
$$\le C s_k^{-1/2}\int (1+(\frac{x-x_k}{s_k})^2)^{-1}\, dx\le C s_k^{1/2}\ .$$
This proves (\ref{modifiedschurc}) since $t_l$ and $t_k$ are
are comparable to $s_l$ and $s_k$, and hence completes the proof of
(\ref{count.tiles}).

If the iterative selection of tents $T_n$ stops because
of lack of suitable points $(y,\eta,t)$ with large enough 
value $F(y,\eta,t)$, then clearly $F$ is bounded by $ \lambda$ outside the union
$\bigcup_{k=1}^{n-1} T_k$. If the iterative selection does not stop, we claim that still
$F$ is bounded  above by $\lambda$  outside the union $\bigcup_{k=1}^\infty T_k$. Namely, assume
the point $(y,\eta,t)$ is outside this union. Since by (\ref{count.tiles})
we have $s_k\to 0$ as $k\to \infty$, we have $s_k<t$ for some
$k$. By maximal choice of $s_k$ we have $F(y,\eta,t)\le  \lambda$.
This proves the desired bound on $F$.
In the case of infinitely many selected tents $T_k$, it also follows
by a limiting argument from (\ref{count.tiles}) that 
$\sum_{k=1}^\infty s_k \le C\lambda^{-2}$ .

Summarizing, we have found a collection $\Q_0$ of tents such that
$$\sum_{T\in \Q_0}^\infty \sigma(T) \le C \lambda^{-2}$$
and if we set
$$E=\bigcup_{T\in \Q_0} T\ ,$$
then we have
$$F(y,\eta,t)\le \lambda$$
for all points $(y,\eta,t)$ in the complement of $E$.
In what follows, we shall no longer need the selected tents 
explicitly, and hence we shall free the symbols $T_k,x_k,\xi_k,s_k$  
to have new meanings in the further selection process. 

We need to select tents of large $L^2$ portion of the size.
Given a number $\xi$, typically arising as second parameter of a tent 
$T(x,\xi,s)$,  we split the space $X$ into upper half 
$$X_\xi^+=\{(y,\eta,s)\in X: \eta\ge \xi\}$$
and lower half $X_\xi^-=X\setminus X_\xi^+$.
We first focus on $X_\xi^+$.

Call a point $(x,\xi,s)\in X_\Delta$ bad, if 
\begin{equation}\label{lowerl2bound} 
s^{-1} \int_{(T(x,\xi,s)\cap X_\xi^+)\setminus (T^b(x,\xi,s)\cup E)}|F(y,\eta,t)|^2\, dy\, d\eta \, {dt}\ge 2^{-8}\lambda^2\ .
\end{equation}
By the estimate (\ref{l2tentbound}) we obtain an a priori
upper bound $2^{k_{\max}}$ for the third component $s$ of any bad point $(x,\xi,s)$. 
Given such an upper bound, the parameter $\xi$ becomes a multiple of 
$2^{-8-k_{\max}}b$ and is thus a discrete parameter. Since $\widehat{f}$ has 
compact support, 
we obtain from observation (\ref{fxibound}) an upper bound for $\xi$ 
depending on the support of $\widehat{f}$. Hence there is a maximal possible 
value $\xi_{\max}$ for the second component of a bad point. 
We choose some bad point $(x_1,\xi_1,s_1)$ with 
$\xi_1=\xi_{\max}$ which maximizes $s_1$ under the constraint $\xi_1=\xi_{\max}$.
Define the tents 
${T}_1=T(x_1,\xi_1,s_1)$ 
and
${T}_1^b=T^b(x_1,\xi_1,s_1)$, and define $X_1^+=X_{\xi_1}^+$.
Note that by maximizing $s_1$ for fixed $\xi_1$ and $x_1$ we guarantee that \eqref{lowerl2bound} 
is sharp up to a factor of $2$ and hence the selected tent satisfies an upper bound
$$ 
s^{-1} \int_{(T(x,\xi,s)\cap X_\xi^+)\setminus (T^b(x,\xi,s)\cup E)}|F(y,\eta,t)|^2\, dy\, d\eta \, {dt}\le \lambda^2\ .
$$

Now we iterate this selection: assume we have already chosen 
points $(x_k,\xi_k,s_k)\in X_\Delta$
for $1\le k<n$ and we have defined tents ${T}_k$, $T_k^b$ 
for $1\le k<n$.
Define $E_n=E\cup \bigcup_{k=1}^{n-1} {T}_k$.
We update the definition of a bad point $(x,\xi,s)$ to be 
a point in $X_\Delta$ with
$$s^{-1}\int_{(T(x,\xi,s)\cap X_\xi^+)\setminus (T^b(x,\xi,s)\cup E_n)} 
|F(y,\eta,t)|^2\, dy d\eta\, dt
\ge 2^{-8}\lambda^2\ .$$
Again, there is a (possibly new) maximal value $\xi_{\max}$
for the second component $\xi$ of a bad point. We pick one bad
point $(x_n,\eta_n,s_n)$ with $\xi_n=\xi_{\max}$ which maximizes
the value of $s_n$ among all bad points $(x,\xi,s)$ with $\xi=\xi_{\max}$.
Then we define the tents ${T}_n=T(x_n,\xi_n,s_n)$ and
${T}_n^b=T^b(x_n,\xi_n,s_n)$
and define $X_n^+=X_{\xi_n}^+$.
This completes the $n$-th selection step.

We introduce the notation
$$ T_n^*= (T_n\cap X_n^+)\setminus (T_n^b\cup E_n)\ .$$

We claim the analogue of (\ref{count.tiles}), namely 
\begin{equation}\label{count.tiles.plus}
\sum_{k=1}^n s_k \le C\lambda^{-2} \ .
\end{equation}
To prove (\ref{count.tiles.plus}), it suffices to show
$$\sum_{k=1}^n \int_{T_k^*}|F(y,\eta,t)|^2 \, dyd\eta dt \le C\ .$$
With $\phi_{y,\eta,t}$ defined analoguously to (\ref{wavepacket}) we
may write for the left hand side of the last display 
$$A:=\sum_{k=1}^n \int_{T_k^*}|\<f,\phi_{y,\eta,t}\>|^2 \, dyd\eta \frac{dt}t \ . $$
Then we have by Cauchy-Schwarz
\begin{equation}\label{lacunarybessel}
A^2 \le \|\sum_{k=1}^n  \int_{{T}_k^*} \<f,\phi_{y,\eta,t}\>\phi_{y,\eta,t}
dy\, d\eta \, \frac{dt}{t} 
\|_2^2
\end{equation}
$$
 =\sum_{k,l=1}^n  \int_{{T}_k^*\times {T}_l^*}
\<f,\phi_{y,\eta,t}\>
\<\phi_{y,\eta,t},\phi_{y',\eta', t'}\>
\<\phi_{y',\eta',t'},f\>
dy\, d\eta \, \frac{dt}{t} dy'\, d\eta' \, \frac{dt'}{t'} 
$$
$$=\sum_{k,l=1}^n  \int_{{T}_k^*\times {T}_l^*: B^{-1}t\le t'\le B t}
\dots 
+2 
\sum_{k,l=1}^n  \int_{{T}_k^*\times {T}_l^*: B t'\le t}
\dots \ ,$$
where the large number $B=2^{8}\alpha^{-1}b^{-1}$ determines the cutoff in the last line between diagonal and
off-diagonal part, the latter being estimated by twice the upper triangular part
using symmetry.  In the diagonal term we use symmetry to estimate the smaller of
the inner products with $f$ by the larger one and obtain the upper
bound
$$2 \sum_{k,l=1}^n  
\int_{{T}_k^*\times {T}_l^*: B^{-1}t \le t'\le B t}
|\<f,\phi_{y,\eta,t}\>|^2
|\<\phi_{y,\eta,t},\phi_{y',\eta', t'}\>|
dy\, d\eta \, \frac{dt}{t} dy'\, d\eta' \, \frac{dt'}{t'} $$
$$\le 2A \sup_{k,(y,\eta,t)\in T_k^*}\left(\sum_{l=1}^n
\int_{{T}_l^*: B^{-1}t\le t'\le B t}
|\<\phi_{y,\eta,t},\phi_{y',\eta', t'}\>| dy'\, d\eta' \, \frac{dt'}{t'} \right)$$
$$\le 2A \sup_{k,(y,\eta,t)\in T_k^*}\left(
\int_{B^{-1}t\le t'\le B t}\int_{\R^2}
|\<\phi_{y,\eta,t},\phi_{y',\eta', t'}\>| dy'\, d\eta' \, \frac{dt'}{t'} \right)\ .$$
Here we have used that the regions $T_l^*$ are pairwise disjoint.
Integrating over the $t'$- interval of bounded $dt'/t'$-measure
estimates the previous display by
$$\le C  A \sup_{k,(y,\eta,t)\in T_k^*} \left(\sup_{B^{-1}t\le t'\le tB}
\int_{\R^2}
|\<\phi_{y,\eta,t},\phi_{y',\eta', t'}\>| dy'\, d\eta'\right)\ . $$
For the $\eta'$ integration we use that
$\widehat{\phi}_{y,\eta,t}$ is supported on an interval of
length $t^{-1}$ and $t\sim t'$:
$$\le C A  \sup_{k,(y,\eta,t)\in T_k^*} \sup_{B^{-1}t\le t'\le tB}
  \sup_{\eta'}\, \,  t^{-1} \int_{\R}
\<|\phi_{y,\eta,t}|,|\phi_{y',\eta', t'}|\> dy' \ .$$
For the $y'$ integration we use that 
${\phi}_{y,\eta,t}$ is an $L^2$ normalized wave
packet adapted to an interval of length $t$. This estimates the last display by $CA$.

Turning to the off diagonal term in (\ref{lacunarybessel}) we estimate it 
with Cauchy Schwarz and the upper bound on the selected tents by
$$2
\sum_{k=1}^n  
\left( 
\int_{{T}_k^*}|\<f, \phi_{y,\eta,t}\>|^2\, dyd\eta \frac{dt}t
\right)^{1/2}
H_k^{1/2}
\le C
\sum_{k=1}^n  \lambda s_k^{1/2} H_k^{1/2}\ ,$$
where $H_k$ is equal to
$$\int_{{T}_k^*}
\left(\sum_{l=1}^n  \int_{{T}_l^*: B t'\le t}
|\<\phi_{y,\eta,t},\phi_{y',\eta', t'}\>
\<\phi_{y',\eta',t'},f\>| dy'\, d\eta' \, \frac{dt'}{t'} 
\right)^{2}
dy\, d\eta \, \frac{dt}{t}\ .$$
Using the pointwise bound on $F(y,\eta,t)$ outside $E$ we can estimate $H_k$ by
$$ \int_{{T}_k^*} 
\left(\sum_{l=1}^n  \int_{{T}_l^*: B t'\le t}
 c\lambda t'^{1/2} |\<\phi_{y,\eta,t},\phi_{y',\eta', t'}\>| dy'\, d\eta' \, \frac{dt'}{t'} 
\right)^{2}
dy\, d\eta \, \frac{dt}{t}\ .$$

Let $(y,\eta,t)\in T_k^*$ and $(y',\eta',t')\in T_l^*$ with $B t'\le t$.
Assume that the inner product $\<\phi_{y,\eta,t},\phi_{y',\eta', t'}\>$ is not zero.
Then we have
$$\eta'+\gamma' (t')^{-1}=\eta+\gamma t^{-1}$$
for some
$$-2^{-8}b \le  \gamma,\gamma' \le 2^{-8}b \ .$$
and hence $$|\eta-\eta'|\le 2^{-4}b(t')^{-1}\ .$$
By definition of the reduced domains we have
$$\alpha(\eta-\xi_k)+\beta t^{-1} \le t^{-1}\ ,$$
$$b(t')^{-1}\le \eta'-\xi_l\ .$$
This gives
$$\xi_k-\xi_l= (\eta-\eta')  - (\eta-\xi_k)+ (\eta'-\xi_l)$$
$$\ge 
-2^{-4}b (t')^{-1}
 - \alpha^{-1} (1-\beta) t^{-1} + b(t')^{-1}
\ .$$
Using $B t'\le t$ the last display  strictly larger than $0$
and hence the tent $T_k$ has been chosen prior to $T_l$.
Since $(y',\eta',t')$ is in the reduced tent  $ T_l^*$, it is
not in $ E_l$ and hence not in $T_k$.
But $t'< t\le s_k$ and 
$$|\alpha(\eta'-\xi_k)+\beta(t')^{-1}|$$
$$\le |\alpha(\eta'-\eta)|+
|\alpha(\eta-\xi_k)+\beta(t)^{-1}|+|\beta (t')^{-1}-\beta(t^{-1})|
$$
$$\le 2^{-4}b(t')^{-1}+t^{-1}+|\beta t^{-1}|+|\beta(t')^{-1}|\le (t')^{-1}$$
and hence we need to have
$$|y'-x_k|\ge s_k-t'\ .$$
This implies
\begin{equation}\label{tentseparation}
|y'-x_k|\ge s_k-t\ .
\end{equation}

Now pick a further point $(y'',\eta'',t'')\in T_{l'}^*$
with $Bt''\le t$ and nonzero inner product 
$\<\phi_{y,\eta,t},\phi_{y'',\eta'', t''}\>$.
We have again 
$$|\eta-\eta''|\le 2^{-4}b(t'')^{-1}$$
and 
$$b(t'')^{-1}\le \eta''-\xi_{l'} \ .$$
Now we assume $Bt''\le t'$.
Then we conclude 
$$|\eta'-\eta''|\le 2^{-2}b(t'')^{-1}$$
and
$$\xi_l-\xi_{l'}=(\xi_l-\eta')+(\eta'-\eta'')+(\eta''-\xi_{l'})$$
$$\ge - 2\alpha^{-1}(t')^{-1} - 2^{-2}b(t'')^{-1}+b(t'')^{-1}>0\ .$$
Hence $T_l$ was chosen prior to $T_{l'}$ and in particular
$(y'',\eta'',t'')$ is not in $T_l$. But we have $t''<t'\le s_l$ and
$$|\alpha(\eta''-\xi_l)+\beta(t'')^{-1}|$$
$$\le |\alpha(\eta''-\eta')|+|\alpha(\eta'-\xi_l)+\beta(t')^{-1}|
+|\beta(t'')^{-1}-\beta(t')^{-1}|
$$
$$\le 2^{-2} b (t'')^{-1}
+(t')^{-1}+ \beta(t'')^{-1}+\beta(t')^{-1} \le (t'')^{-1}\ .$$
Since $(y'',\eta'',t'')$ is not in $T_l$ 
we conclude
$$|y''-x_l|>s_l-t''>s_l-t'\ge |y'-x_l|$$
and in particular $y''\neq y'$.

%
%

To summarize our finding, fix $(y,\eta,t)\in T_k^*$.
Then for fixed $y'\in \R$, the minimal and maximal values
of parameters $t'$ with $Bt'\le t$ such that there exists 
$l$ and $\eta'$ with $(y',\eta',t')\in T_l^*$ and 
$\<\phi_{y,\eta,t},\phi_{y',\eta', t'}\>\neq 0$ are 
at most a factor $B$ apart. 
It follows that for every $y'$ there exists an interval $I(y')=[T(y'), B T(y')]$
such that we need $t'\in I(y')$ for such $l,\eta'$ to exist.

Using also (\ref{tentseparation}) and disjointness of the reduced domains
$T_l^*$ , we may thus estimate $H_k$ by
$$C \int_{{T}_k^*} 
\left( \int_{|y'-x_x|>s_k-t}  
\int_ {I(y')} \int_\R 
 \lambda t'^{1/2} |\<\phi_{y,\eta,t},\phi_{y',\eta', t'}\>| \, d\eta'
\, \frac{dt'}{t'} 
 \, dy' \right)^{2}
dy\, d\eta \, \frac{dt}{t}\ .$$
$$ \le C\int_{{T}_k^*} 
\left(
\int_{|y'-x_x|>s_k-t}  
 \sup_{t'\in I(y')}\int_\R 
 \lambda t'^{1/2} |\<\phi_{y,\eta,t},\phi_{y',\eta', t'}\>|  
\, d\eta' \, dy' \right)^{2}
dy\, d\eta \, \frac{dt}{t}\ .$$

Further, by trivial reasoning with the Fourier support of the
bump functions, if we fix $y'$ and $t'$ as in this integral, 
then there is an interval of length $2t'^{-1}$ which must contain $\eta'$
for the inner product $\<\phi_{y,\eta,t},\phi_{y',\eta', t'}\>$ to
be nonzero. Using the estimate 
$$\<|\phi_{y,\eta,t}|,|\phi_{y',\eta', t'}|\>   \le C (\frac{t'}t)^{1/2}
(1+\frac{|y'-y|}{t})^{-2}\ ,$$
we obtain for the previous display the upper bound  
$$  C\int_{T_k^*} 
\left( \int_{|y'-x_k|>s_k-t}\
 \lambda (\frac{1}{t})^{1/2} (1+\frac{|y'-y|}{t})^{-2}
 dy' \, 
\right)^{2}
dy\, d\eta \, \frac{dt}{t} $$
$$\le  C\int_{T_k^*} 
\left( 
 \lambda {t}^{1/2} (1+\frac{s_k-|y-x_k|}{t})^{-1}
\right)^{2}
dy\, d\eta \, \frac{dt}{t} $$
$$\le  C\int_0^{s_k} \int_{x_k-s_k}^{x_k+s_k}\int_{\xi_k-2\alpha^{-1}t^{-1}}^{\xi_k+2\alpha^{-1}t^{-1}} 
 \lambda^2 t(1+\frac{s_k-|y-x_k|}{t})^{-2}
d\eta \, dy\, \frac{dt}{t} $$
$$\le C\lambda^2s_k\ .$$
This completes our estimation of (\ref{lacunarybessel})
and we have shown
$$\frac 14 A^2\le CA + C\sum_{k=1}^n\lambda^2 s_k\le CA\ ,$$
where in the last inequality we have used the lower bound
on the selected tents.
Dividing by $A$ proves the desired estimate for $A$ and
completes the proof of (\ref{count.tiles.plus}) for the newly selected tents.

If the selection of tents stops lacking any further $(x,\xi,s)$ with
$$s^{-1}\int_{(T(x,\xi,s)\cap X_\xi^+)\setminus (T^b_{x,\xi,s}\cup E_n)} 
|F(y,\eta,t)|^2\, dy\, d\eta\, dt
\ge 2^{-8}\lambda^2\ ,$$
then clearly the converse inequality holds for all $(x,\xi,s)\in X_\Delta$.
If the selection of tents does not stop, we collect
$T_k$ for all $k\in \N$ and write $E_{(1)}=E\cup_{k=1}^\infty E_k$.
Note that $\xi_k$ is a decreasing sequence, and as noted
before the possible values of $\xi_k$ are in the discrete lattice
$\Z b 2^{-8-k_{\max}}$.
If $\xi_k\to -\infty$, then for every $(x,\xi,s)\in X_\Delta$
$$s^{-1}\int_{(T(x,\xi,s)\cap X_\xi^+)\setminus (T^b_{x,\xi,s}\cup E_{(1)})} 
|F(y,\eta,t)|^2\, dy\, d\eta\, dt
\le 2^{-8}\lambda^2\ .$$
Namely, assume not, then $\xi> \xi_k$ for some $k$, and this
would contradict the choice of $T_k$.

Now assume $\xi_k$ does not tend to $-\infty$, then the sequence
stabilizes, that means eventually becomes constant, at some 
value $\xi_{(1)}$. We shall then choose further
tents, and for emphasis we rename the previously selected tents
into $ T_k =:T_{(1),k}=T(x_{(1),k},\xi_{(1),k},s_{(1),k})$.

Call a point $(x,\xi,s)\in X_\Delta$ bad if
$$s^{-1}\int_{T(x,\xi,s)\cap X_\xi^+)\setminus (T^b(x,\xi,s)\cup E_{(1)})} 
|F(y,\eta,t)|^2\, dy\, d\eta\, dt
\ge 2^{-8}\lambda^2\ ,$$
and let $\xi_{\max}$ be the maximal possible value of the second
component of a bad point. Note that $\xi_{\max}$ is strictly
less than $\xi_{(1)}$. For if not, then $\xi_{\max}=\xi_{(1)}$
by choice of the previously selected tents. Since $s_k\to 0$
by (\ref{count.tiles.plus}) we have $s>s_k$ for some $k$ and some
bad point $(x,\xi_{\max},s)$. This
however contradicts the choice of the tent $T_{(1),k}$.
We then choose a bad point $(x_{(2),1},\xi_{(2),1},s_{(2),1})$
such that $\xi_{(2),1}=\xi_{\max}$ and $s_{(2),1}$ is maximal among
all such choices.
We then iterate this selection process as before, obtaining
tents $T_{(2),k}=T(x_{(2),k},\xi_{(2),k},s_{(2),k})$. 
Our proof of \eqref{count.tiles.plus} applies verbatim to yield
$$(\sum_{k=1}^\infty s_{(1),k})+ (\sum_{k=1}^n s_{(2),k}) 
\le C\lambda^{-2} \ .$$

We now continue this double recursion in the obvious manner. 
If at some point the recursion stops,  or yields for some fixed 
$m$ a sequence $\xi_{(m),k}$ tending to $-\infty$, then by the
previous discussions we are left with no bad points. If the 
double iteration does not stop, we obtain a double sequence of
tents $T_{(m),k}$ with
$$\sum_{m=1}^\infty \sum_{k=1}^\infty s_{(m),k}\le C\lambda^{-2} \ .$$
Moreover, the sequence $\xi_{(m)}$ of stabilizing points decreases
to $-\infty$, since they are strict monotone decreasing and 
in a discrete lattice. We can then observe that there are no
bad points outside $\bigcup _{m=1}^\infty E_{(m)}$.

Summarizing, we have found 
a collection $\Q_+$ of tents such that
$$\sum_{T\in \Q_+}^\infty \sigma(T) \le C \lambda^{-2}$$
and if we set
$$E_+=E\cup \bigcup_{T\in \Q_+} T\ ,$$
then we have
$$s^{-1}\int_{(T(x,\xi,s)\cap X_\xi^+)\setminus T^b(x,\xi,s)} 
|F(y,\eta,t)1_{ E_+^c}(y,\eta,t)|^2\, dy\, d\eta\, dt
\le 2^{-8}\lambda^2\ ,$$
for all $(x,\xi,\delta)\in X_\Delta$.

We may repeat the above argument symmetrically to obtain 
a collection $\Q_-$ of tents such that
$$\sum_{T\in \Q_-}^\infty \sigma(T) \le C \lambda^{-2}$$
and if we set
$$E_-=E\cup \bigcup_{T\in \Q_-} T\ ,$$
then we have
$$s^{-1}\int_{(T(x,\xi,s)\cap X_\xi^-)\setminus T^b(x,\xi,s)} 
|F(y,\eta,t)1_{E_-^c}(y,\eta,t)|^2\, dy\, d\eta\, dt
\le 2^{-8}\lambda^2\ ,$$
for all $(x,\xi,\delta)\in X_\Delta$.

Setting finally $\Q=\Q_0\cup \Q_+\cup \Q_-$ we have
clearly found the desired collection of tents.
This completes the proof 
of the endpoint $p=2$ of Theorem \ref{gen.carl.emb}.
\end{proof}

\section{The bilinear Hilbert transform}
\label{bhtsection}

The most immediate application of Theorem \ref{gen.carl.emb} is to prove basic
estimates for the bilinear Hilbert transform. Another possible application is
towards Carleson's theorem \cite{carleson} on almost everywhere convergence of 
Fourier series. However, the latter application requires more work, as 
Carleson's operator lacks the symmetry that is exhibited by the bilinear Hilbert 
transform and therefore needs an additional embedding theorem. Hence we decided
to restrict attention to the bilinear Hilbert transform, which suffices to 
illustrate some key points of time-frequency analysis originating in Carleson's work
on convergence of Fourier series.

Let $\beta=(\beta_1,\beta_2,\beta_3)$ be a vector in $\R^3$ with pairwise
distinct entries. For three Schwartz functions $f_1,f_2,f_3$ on the real line  
we define
$$\Lambda_\beta(f_1,f_2,f_3):=p.v. \int_\R [\int_\R [\prod_{j=1}^3 f_j(x-\beta_jt) ]\, dx ]\, \frac{dt} t\ .$$
Note that the inner integral produces a Schwartz function in the variable
$t$, to which we apply the tempered distribution $p.v. 1/t$.
By a change of variables, scaling $t$ and translating $x$, we may and do restrict 
attention to vectors $\beta$ which have unit length and are perpendicular to 
$(1,1,1)$. The resulting one parameter family of trilinear forms is dual
to a family of bilinear operators called bilinear Hilbert transforms.
To obtain explicit expressions for these bilinear operators, one applies 
another translation in the $x$ variable to make one of the components, 
say $\beta_i$ vanish. After interchanging 
the order of integrals one obtains an explicit pairing of a bilinear 
operator in $f_j$, $j\neq i$, with the function $f_i$.

Let $\alpha$ be a unit vector perpendicular to $(1,1,1)$ and $\beta$.
The vector $\alpha$ is unique up to reflection at the origin, and has
only non-zero components by the assumption that $\beta$ has
pairwise distinct components. Note also that $|\beta_j|\le 0.9$ for each
$j$. For if one component of $\beta_j$ in absolute value exceeds $0.9$, then 
since $\beta$ is perpendicular to $(1,1,1)$, at least one further component 
has to exceed $0.45$ in absolute value. But then the vector cannot be a unit vector.


The following a priori 
estimate for $\Lambda_\beta$ originates in \cite{lacey-thiele1}.

\begin{theorem}\label{bilinearhttheo}
For a unit vector $\beta$ perpendicular to $(1,1,1)$ with pairwise
distinct entries, and for 
$2<p_1,p_2,p_3<\infty$ with $\sum_j \frac 1{p_j} =1$, there is a constant 
$C$ such that for all Schwartz functions $f_1,f_2,f_3$ we have
$$|\Lambda_\beta (f_1,f_2,f_3)|\le C \prod_{j=1}^3 \|f_j\|_{p_j}\ .$$ 
\end{theorem}

We give a new proof of this theorem based on 
Theorem \ref{gen.carl.emb}
and an outer H\"older inequality. This proof is analoguous to the 
previously presented proof of boundedness of paraproducts.
In our approach, much of the difficulty in proving bounds for the 
bilinear Hilbert transform has been moved into the proof of the
generalized Carleson embedding theorem. What remains to be done
is relatively easier and in particular conceptually quite simple. 
It is the strength of our approach that the main difficulty
is packaged into a cleanly separated module; previous approaches
do not suggest the formulation of as clean a statement as 
Theorem \ref{gen.carl.emb}. In particular, our proof is the first one
to succeed without the passage to a discrete model operator. This
avoids a cumbersome setup of choices of the discretization.

One can prove a version of Theorem \ref{bilinearhttheo} with 
a constant independent of $\beta$, see \cite{grafakos-li}, but 
only at the expense of considerable additional work. 
One may also extend the range of exponents, see \cite{lacey-thiele2}. 
It would be interesting to discuss these results in the context of
outer measure theory, but this is beyond the scope of the present 
paper.

\begin{proof}[Proof of Theorem \ref{bilinearhttheo}]
Define for
$j=1,2,3$
\begin {equation}\label{bhtemb}
F_j(y,\eta,t):=\int_\R f_j(x) e^{i\eta (y-x)} t^{-1}\phi(t^{-1}(y-x)) \, dx\ ,
\end{equation}
where $\phi$ is a real valued Schwartz function such that $\widehat{\phi}$
is nonnegative, non vanishing at the origin, and supported in 
$[-\epsilon,\epsilon]$ for suitably small $\epsilon$. It will suffice to choose 
$\epsilon=2^{-16}$.

The estimate of Theorem \ref{bilinearhttheo} can be reformulated
by means of the functions $F_j$. 
\begin{lemma}\label{bhttheorem}
Under the assumptions of Theorem \ref{bilinearhttheo}
there is a constant $C$ depending only on $\beta$, $p_1,p_2,p_3$ 
and $\phi$ as above such that
\begin{equation}\label{bhtbound}
\left|\int_0^\infty \int_\R\int_\R 
\prod_{j=1}^3 F_j (y,\alpha_j\eta +\beta_jt^{-1}, t) \, d\eta \, dy\, dt\right|
\le C\prod_{j=1}^3 \|f_j\|_{p_j}\ .
\end{equation}
\end{lemma}

We postpone the proof of Lemma \ref{bhttheorem} and proceed to deduce
Theorem \ref{bilinearhttheo} from Lemma \ref{bhttheorem}. 

Inserting the definition of $F_j$  and 
using that $\alpha$ and $ \beta$ are perpendicular to $(1,1,1)$ we 
obtain for the integral on the left-hand-side of (\ref{bhtbound}):
$$\int_0^\infty \int_\R\int_\R 
t^{-3}\prod_{j=1}^3 
[\int_\R f_j(x_j) e^{-i \beta_j t^{-1}x_j} e^{-i \alpha_j\eta x_j} 
\phi(t^{-1}(y-x_j)) \, dx_j]
\, d\eta \, dy\, dt\ .$$
Recall that the integral of the Fourier transform of a Schwartz function 
$\varphi$ in $\R^3$ over the  
line through the origin spanned by $\alpha$ is proportional to the integral 
of the Schwartz function itself over the perpendicular hyperplane through the origin spanned by $(1,1,1)$ and $\beta$:
\begin{equation}\label{perpendicularintegrals}
\int_\R \widehat{\varphi}(\eta \alpha)\, d\eta=
c \int_\R \int_\R \varphi(u(1,1,1)+v \beta)\, du \, dv\ .
\end{equation}
To apply this fact, we observe that the inner triple integral of the 
previous display over $x_1,x_2,x_3$ is the value of 
the Fourier transform of a certain Schwartz function in $\R^3$ at the point 
$\eta\alpha \in \R^3$, and the integral in $\eta$ is then the integral of 
this Fourier transformation over the line spanned by $\alpha$. 
Hence we obtain up to a nonzero constant factor for that display:
$$
\int_0^\infty \int_\R 
\int_\R\int_\R[t^{-3}    e^{- i t^{-1} v} \prod_{j=1}^3 
f_j(u+\beta_jv)  \phi(t^{-1}(y-u-\beta_j v))]  \, du\, dv 
\, dy\, dt\ .$$
Here we have used again in the argument of the exponential function that $\alpha$, $\beta$  and
$(1,1,1)$ are pairwise orthogonal and that $\beta$ has unit length.
Changing the order of integration so that the $y$ integration becomes innermost
we obtain for the last display 
$$
\int_{0}^\infty \int_\R \int_\R [\prod_{j=1}^3 f_j(u+\beta_j v)]
t^{-2}e^{-it^{-1} v}
\psi(t^{-1}v)\, du \, dv\, {dt}\ , 
$$
where 
$$\psi (w):= \int_\R 
\prod_{j=1}^3 \phi(z-\beta_j w)\, dz\ .$$
We claim that there are nonzero constants $a$ and $b$ such that
for any Schwartz function $g$ on the real line we have
\begin{equation}\label{lincombclaim}
\int_{0}^\infty \int_\R  g(v)t^{-2}e^{-it^{-1} v} \psi(t^{-1}v)\, dv \, dt=
a g(0)+b\, p.v. \int g(t)\frac{dt}t\ .
\end{equation}
This claim turns the left hand side of
(\ref{bhtbound}) into a nontrivial linear combination of 
$$\int_\R [\prod_{j=1}^3 f_j(u)]\, du$$
 and 
$$p.v. \int[\int f_1(u-\beta_1 t) f_2(u-\beta_2 t) f_3(u-\beta_3 t)  \, du ] \frac {dt} t\ . $$
Since $L^p$ bounds for the former follow by H\"older's inequality, we can deduce
$L^p$ bounds for the latter from $L^p$ bounds as in  (\ref{bhtbound}). 
This will complete the reduction of Theorem \ref{bilinearhttheo}
to Lemma \ref{bhttheorem}, 
once we have verified the above claim.

To see the claim, it suffices to verify that the left-hand-side of
(\ref{lincombclaim}) can be written as a nonzero multiple of 
$$\int_{-\infty}^0 \widehat{g}(\zeta)\, d\zeta\ ,$$
since the characteristic function of the left half line is known to be 
a nontrivial linear combination
of the Fourier transform of the Dirac delta distribution and
the principal value integral against $dt/t$.
Using Plancherel we identify the left-hand-side of (\ref{lincombclaim})
as nonzero multiple of
$$\int_{0}^\infty \int_\R  \widehat{g}(\zeta) \widehat{\psi}( 1-t\zeta )\, d\zeta \, \frac{dt}t\ .$$

The claim will thus follow by Fubini if we can establish that $\widehat{\psi}$   is proportional to a function that 
is nonnegative, nonzero at $0$, and supported in $[-1/2, 1/2]$.
We have 
$$\widehat{\psi}(\eta)= \int_\R\int_\R 
\prod_{j=1}^3 \phi(z-\beta_j w) e^{i  \beta_j \eta (z-\beta_j w)} \, dz\, dw .$$
This is an integral of a Schwartz function in $\R^3$ over the plane spanned 
by $(1,1,1)$ and $\beta$, which by the observation (\ref{perpendicularintegrals}) again may be written
as multiple of the integral
of the Fourier transform of the Schwartz function over the line spanned by $\alpha$:
$$\int_\R \prod_{j=1}^3 \widehat{\phi}(\alpha_j \xi-\beta_j \eta))\, d\xi\ \ .$$
Since $\alpha$ and $\beta$ are perpendicular unit vectors and the support of
$\widehat{\phi}\otimes \widehat{\phi} \otimes \widehat{\phi}$ is in a neighborhood 
of $0$ with diameter less than $1/10$, this integral is non-zero only if $|\eta|$ is smaller than $1/2$. Moreover,
$\widehat{\psi}$ is evidently nonnegative real and nonzero at $0$.
This completes the proof of the claim and the reduction of Theorem \ref{bilinearhttheo}
to Lemma \ref{bhttheorem}.
\end{proof} 

\begin{proof}[Proof of Lemma \ref{bhttheorem}]
We consider the space $X=\R\times \R\times (0,\infty)$ and the outer measure generated by the collection $\E$ of all tents 
$$T(x,\xi,s):=\{(y,\eta,t)\in X: t<s, |y-x|<s-t, |\eta-\xi|\le t^{-1}\} $$
parameterized by $(x,\xi,s)\in X$  and the premeasure 
$\sigma(T(x,\xi,s))=s$.

Define a  size $S$ by setting
$$S(G)(T(x,\xi,s))=s^{-1}\int_{T(x,\xi,s)} |G(y,\eta,t)|\, dy\, d\eta\, dt$$
for each $G\in \B(X)$.

By a straight forward application of Proposition \ref{measuredomination} we may estimate the left-hand-side
of (\ref{bhtbound}) by
$$C\|G_1G_2G_3\|_{L^1(X,\sigma,S)}\ ,$$
where we have defined $G_j$  for $j=1,2,3$ by
$$G_j(y,\eta,t):=F_j(y,\alpha_j \eta+\beta_j t^{-1}, t)\ .$$

We intend to apply a threefold H\"older's inequality, which requires us to define three appropriate sizes $S_j$.
Set
$$b=2^{-8}\min_{i\neq j}|\beta_i-\beta_j|\ .$$
Since no two components of $\beta$ are equal, we have $b>0$. 
Define for each $1\le j\le 3$ and $(x,\xi,s)\in X$ the region 
$$T^{(j)}(x,\xi,s)$$
$$:=\{(y,\eta,t)\in X: t \le  s, |y-x|  \le  s-t, |\alpha_j^{-1}(\eta-\xi)   -  \alpha_j^{-1}\beta_jt^{-1}|\le b t^{-1}\}\ .$$

For fixed $(x,\xi,s)$ the three regions 
$T^{(j)}(x,\xi,s)$ are pairwise disjoint, by symmetry it suffices to establish this
for $j=1,2$. Assume to get a contradiction that we have $\eta,t$ with
$$|\alpha_1^{-1}(\eta-\xi)   -   \alpha_1^{-1}\beta_1 t^{-1}|,\  |\alpha_2^{-1}(\eta-\xi)  -   \alpha_2^{-1}\beta_2 t^{-1}|\le 
b t^{-1}\ .$$
Multiplying by $|\alpha_1|,|\alpha_2|\le 1$ respectively and comparing yields
$|\beta_1-\beta_2|\le 2b$.
This however is a contradiction to the choice of $b$ and thus proves that the regions 
$T^{(j)}(x,\xi,s)$ are pairwise disjoint.

We now observe for each $T=T(x,\xi,s)$ with similar notation
$T^{(j)}= T^{(j)}(x,\xi,s)$ 
$$s S(G)(T)
= \int_{T} |G(y,\eta,t)|\, dy\, d\eta\, dt$$
$$=
\int_{T\setminus(T^{(1)}\cup T^{(2)}\cup T^{(3)})}
|G(y,\eta,t)|\, dy\, d\eta\, dt
+\sum_{j=1}^3 \int_{T\cap T^{(j)}}|G(y,\eta,t)|\, dy\, d\eta\, dt$$
$$\le \prod_{j=1}^3\left( 
\int_{T\setminus T^{(j)}}
|G_j(y,\eta, t)|^3\, dy\, d\eta\, dt \right)^{1/3}$$
$$+\sum_{j=1}^3
\sup_{( y,\eta,t  )\in T^{(j)}}|G_j(y,\eta,t)|
\prod_{k\neq j}
\left(\int_{T\setminus T^{(k)}}|G_k(y,\eta,t)|^2\, dy\, d\eta\, dt \right)^{1/2}\ .$$
Define the size 
$$S_j(G)(T):=(
s^{-1}\int_{T\setminus T^{(j)} }|G(y,\eta,t)|^2 \, dy\, 
d\eta  \,{dt})^{1/2}
+ \sup_{(y,\eta,t)\in T}|G(y,\eta,t)|\ .$$
Then we conclude from the previous considerations that 
$$S(G)(T)
\le 4 \prod_{k=1}^3 S_k(G_k)(T)\ ,$$
where we have with log convexity estimated $L^3$ norms by $L^2$ and $L^\infty$ norms.

By the outer H\"older inequality, Proposition \ref{p.hoelder-energy},  
we obtain for the left-hand-side of (\ref{bhtbound}) the bound 
$$C \prod_{j=1}^3 \|G_j\|_{L^{p_j}(X,\sigma, S_j)} $$
with exponents $p_j$ as in Lemma \ref{bhttheorem}.
It remains to show for each $j$ that 
$$\|G_j\|_{L^{p_j}(X,\sigma,S_j)}\le C \|f_j\|_{p_j}\ .$$
This follows from the generalized Carleson embedding, Theorem \ref{gen.carl.emb}, 
after a re-parametrization of the space $X$ under the homeomorphism
$$\Phi_j:X\to X,\  (y,\eta,t)\mapsto (y,\alpha_j \eta+\beta_j t^{-1}, t)\ .$$
Note that $\Phi_j$ maps 
$T_{\alpha_j,\beta_j}(x,\alpha_j^{-1} \xi,s)$ as defined in \eqref{gentent}
to $T(x,\xi,s)$ as above, and it maps
$T^b(x,\alpha_j^{-1}\xi,s)$ to $T^{(j)}(x,\xi,s)$ as above
and we have $F_j\circ\Phi_j=G_j$.
This completes the proof of Lemma \ref{bhttheorem}.

\end{proof}




\end{document}